\documentclass[a4paper,11pt]{article}
\usepackage{amsmath,amsthm,amssymb,enumitem}

\usepackage{tikz}

\usepackage[nosort,nocompress,noadjust]{cite}

\usepackage[linktocpage=true,bookmarks=false,hyperfootnotes=false,colorlinks,
    linkcolor={red!60!black},
    citecolor={blue!50!black},
    urlcolor={blue!80!black}]{hyperref}

\renewcommand{\eqref}[1]{\hyperref[#1]{(\ref{#1})}}

\newlist{enumlist}{enumerate}{1}
\setlist[enumlist]{labelindent=0cm,label=\arabic*.,labelwidth=2.5ex,labelsep=0.5ex,leftmargin=3ex,align=left,topsep=0.5ex,itemsep=1ex,parsep=1ex}

\newlist{itemlist}{itemize}{1}
\setlist[itemlist]{labelindent=0cm,label=$\bullet$,labelwidth=2.5ex,labelsep=0.5ex,leftmargin=3ex,align=left,topsep=0.5ex,itemsep=1ex,parsep=1ex}

\numberwithin{equation}{section}

{\theoremstyle{definition}\newtheorem{definition}{Definition}[section]

\newtheorem{remark}[definition]{Remark}
}

\newtheorem{proposition}[definition]{Proposition}
\newtheorem{lemma}[definition]{Lemma}
\newtheorem{theorem}[definition]{Theorem}
\newtheorem{corollary}[definition]{Corollary}

\newtheorem{letterthm}{Theorem}

\newcommand{\C}{\mathbb{C}}
\newcommand{\cC}{\mathcal{C}}
\newcommand{\eps}{\varepsilon}
\newcommand{\al}{\alpha}
\newcommand{\be}{\beta}

\newcommand{\ot}{\otimes}
\newcommand{\recht}{\rightarrow}

\newcommand{\Z}{\mathbb{Z}}
\newcommand{\vphi}{\varphi}

\newcommand{\id}{\mathord{\text{\rm id}}}
\newcommand{\om}{\omega}

\newcommand{\N}{\mathbb{N}}

\newcommand{\ovt}{\mathbin{\overline{\otimes}}}

\newcommand{\cD}{\mathcal{D}}
\newcommand{\si}{\sigma}

\newcommand{\F}{\mathbb{F}}

\newcommand{\cZ}{\mathcal{Z}}
\newcommand{\Ad}{\operatorname{Ad}}

\newcommand{\cG}{\mathcal{G}}

\newcommand{\cF}{\mathcal{F}}
\newcommand{\T}{\mathbb{T}}
\newcommand{\actson}{\curvearrowright}
\newcommand{\cS}{\mathcal{S}}

\newcommand{\cB}{\mathcal{B}}

\newcommand{\cU}{\mathcal{U}}

\newcommand{\cN}{\mathcal{N}}
\newcommand{\cR}{\mathcal{R}}

\newcommand{\dpr}{^{\prime\prime}}

\newcommand{\cV}{\mathcal{V}}

\newcommand{\cE}{\mathcal{E}}

\newcommand{\Aut}{\operatorname{Aut}}
\newcommand{\Inn}{\operatorname{Inn}}
\newcommand{\cP}{\mathcal{P}}

\newcommand{\cT}{\mathcal{T}}
\newcommand{\VC}{\mathcal{VC}}
\newcommand{\cQ}{\mathcal{Q}}
\newcommand{\Q}{\mathbb{Q}}
\newcommand{\VCCartan}{\mathcal{VC}_{\text{\rm Cartan}}}
\newcommand{\cost}{\operatorname{cost}}
\newcommand{\M}{\mathbb{M}}

\begin{document}

\begin{center}
{\boldmath\LARGE\bf Classification of regular subalgebras \vspace{0.5ex}\\ of the hyperfinite II$_1$ factor}

\vspace{1ex}

{\sc by Sorin Popa\footnote{Mathematics Department, UCLA, Los Angeles, CA 90095-1555 (United States), popa@math.ucla.edu\\
Supported in part by NSF Grant DMS-1700344}, Dimitri Shlyakhtenko\footnote{Mathematics Department, UCLA, Los Angeles, CA 90095-1555 (United States), shlyakht@math.ucla.edu\\ Supported in part by NSF Grant DMS-1500035} and Stefaan Vaes\footnote{KU~Leuven, Department of Mathematics, Leuven (Belgium), stefaan.vaes@kuleuven.be \\
    Supported in part by European Research Council Consolidator Grant 614195, and by long term structural funding~-- Methusalem grant of the Flemish Government.\vspace{0.3ex}\\ \mbox{}\hspace{6mm}Part of this research was  completed while the authors were visiting the Institute for Pure and Applied Mathematics (IPAM), which is supported by the National Science Foundation.}}

\vspace{1ex}

{\it Dedicated to Alain Connes.}
\end{center}

\vspace{1ex}

\begin{abstract}\setlength{\parindent}{0pt}\setlength{\parskip}{1ex}\noindent
We prove that the regular von Neumann subalgebras $B$ of the hyperfinite II$_1$ factor $R$ satisfying the condition $B'\cap R=\mathcal Z(B)$ are completely classified
(up to conjugacy by an automorphism of $R$) by the associated discrete measured groupoid $\cG=\cG_{B\subset R}$. We obtain a similar classification result for  triple inclusions $A\subset B \subset R$, where $A$ is a Cartan subalgebra in $R$ and the intermediate von Neumann algebra $B$ is regular in $R$.
A key step in proving these results is to show the vanishing cohomology  for the associated cocycle actions $(\alpha_{B\subset R}, u_{B\subset R})$ of $\cG$ on $B$. We in fact prove two very general vanishing cohomology results for free cocycle actions $(\alpha, u)$ of amenable discrete measured groupoids $\cG$ on {\it arbitrary} tracial von Neumann algebras $B$, resp. Cartan inclusions $A\subset B$. Our work provides a unified approach and generalizations to many known vanishing cohomology and
classification results \cite{CFW81}, \cite{O85}, \cite{Suth-Tak84}, \cite{BG84}, \cite{FSZ88}, \cite{Popa18}, etc.
\end{abstract}

\section{Introduction}

Connes's fundamental theorem (\cite{C75}), establishing  the equivalence between amenability and finite dimensional approximation (AFD) for finite von Neumann algebras, has two remarkable consequences. On the one hand, it gives a complete classification of the von Neumann subalgebras of the hyperfinite II$_1$ factor $R$, showing that they are all AFD of the form $B=\oplus_{n\geq 1} (A_n \otimes \M_{n\times n}(\C)) \oplus (A_0 \ovt R)$, where $A_n$, $0\leq n <\infty$, are separable and abelian, possibly equal to  $0$.
On the other hand, it shows that any crossed product von Neumann algebra $B\rtimes_{\alpha, u} \cG$,
corresponding to a free cocycle action $(\alpha, u)$ of an amenable discrete groupoid $\cG$ on
a diffuse tracial AFD algebra $B$, which is ergodic on the center $\cZ(B)$ of $B$, is isomorphic to $R$.

During the late 1970s and early 1980s, much effort has been put into clarifying whether
such a crossed product decomposition of $R$  is in fact uniquely determined by the nature of $B$ and $\cG$ alone (thus not depending
on $(\alpha, u)$).
Alternatively, the question is whether any two regular copies $B\subset R$ of a specific AFD algebra $B$,
with $B'\cap R = \cZ(B)$ and with the same associated amenable discrete measured groupoid $\cG_{B\subset R}=\cG$, are
conjugate by an automorphism  of $R$. The
assumptions imply that $B$ is necessarily homogeneous of the form $B= L^\infty(X, \mu) \ovt N$, where either $N\simeq \M_{n\times n}(\C)$, for some $n\geq 1$,
or $N\simeq R$.
The normalizer $\cN_R(B):=\{u\in \cU(R)\mid uBu^*=B\}$ defines an amenable discrete measured groupoid $\cG=\cG_{B\subset R}$ together with a free cocycle action $(\alpha, u)$ of $\cG$ on $B$ (see Section 2 for a rigorous definition and detailed discussion of discrete measured groupoids and their free
cocycle actions on tracial von Neumann algebras). The groupoid $\cG$ accounts for an amenable ergodic countable equivalence relation ``along'' the space $\cG^{(0)}=X$ of units of $\cG$, with amenable countable isotropy groups $\Gamma_x$ at each $x\in X$ acting outerly on $B_x\simeq N$. When $B$ is abelian, then $B \simeq L^\infty(X, \mu)$ and one calls it a {\it Cartan subalgebra} of $R$. In this case, $\cG$ is just a countable equivalence relation $\cR$ on $X$, with $\alpha$ intrinsic to $\cR$. If $B$ is a factor, then $B\simeq R$ and
$\cG_{B\subset R}$ coincides with the group $\Gamma=\cN_R(B)/\cU(B)$, which is automatically
countable amenable, with $(\alpha, u)$ the free cocycle action of $\Gamma$ on $B$
implemented by $\cN_R(B)$.

The uniqueness problem  has been solved in two important cases: when $B$ is abelian,
i.e., for Cartan subalgebras of $R$ (\cite{CFW81}); and when $B$ is a factor, i.e., when $B\simeq R$ (\cite{O85}).
Thus, it was shown  in \cite{CFW81} that any two Cartan subalgebras of $R$ are conjugate by an automorphism of $R$. Equivalently, there is a unique
amenable ergodic type II$_1$ equivalence relation, which has vanishing 2-cohomology. This also implies that
for any $n\geq 2$, any two regular subalgebras of type I$_n$ of $R$ are conjugate by an automorphism of $R$.
In turn, the result in \cite{O85} shows hat
any two free cocycle actions of the same countable amenable group $\Gamma$ on $B \simeq R$ are cocycle conjugate. Equivalently, given any countable amenable
group $\Gamma$, there exists a unique (up to conjugacy by an automorphism of $R$) irreducible regular subfactor $B \subset R$
with $\cN_R(B)/\cU(B)\simeq \Gamma$.

An important step
towards solving the remaining case when $B$ is an arbitrary II$_1$ AFD algebra, $B \simeq L^\infty(X, \mu) \ovt R$,
has been achieved in \cite{Suth-Tak84}, where it was shown
that any two free \emph{genuine} actions $\alpha$ of the same amenable discrete measured groupoid $\cG$ on $B$ that are ergodic on
$\cZ(B)=L^\infty(X)=L^\infty(\cG^{(0)})$, are cocycle conjugate. However, this result does not settle
the uniqueness of all {\it cocycle actions} $(\alpha, u)$ of $\cG$ on arbitrary such $B$.

We solve this last step here, by proving that in fact any such cocycle $u$ untwists. When combined with \cite{Suth-Tak84}, this shows the uniqueness
of all cocycle actions of $\cG$ on $B$. Equivalently, all regular  embeddings  $B\subset R$, with $B'\cap R=\cZ(B)$, of a given II$_1$
AFD algebra $B$
that have the same groupoid $\cG_{B\subset R}\simeq \cG$ are conjugate by an automorphism of $R$.

We in fact prove a very general vanishing cohomology result, showing that any free cocycle action $(\alpha, u)$
of any amenable discrete measured groupoid $\cG$ on {\it any} (not necessarily AFD)
tracial von Neumann algebra $(B, \tau)$, untwists. We also prove a relative version of this result, by showing that
if in addition $(\alpha, u)$ is assumed to normalize a Cartan subalgebra $A$ of $B$, on which it acts freely, then
the vanishing cohomology
can be realized within the normalizer of $A$ in $B$, i.e., as a coboundary of a map from $\cG$ into the normalizer of $A$ in $B$.
More precisely, in Theorems \ref{thm.two-cocycle-vanishing} and \ref{thm.two-cocycle-vanishing-Cartan}, we show:

\begin{letterthm}\label{thm-A.two-cocycle-vanishing}
Let $\cG$ be a discrete measured groupoid with $X = \cG^{(0)}$ and $(B_x)_{x \in X}$ a measurable field of II$_1$ factors with separable predual. Assume that $\cG$ is amenable and that $(\al,u)$ is a free cocycle action of $\cG$ on $(B_x)_{x \in X}$.
\begin{enumlist}
\item The cocycle $u$ is a co-boundary: there exists a measurable field of unitaries $\cG \ni g \mapsto w_g \in B_{t(g)}$ such that $u(g,h) = \al_g(w_h^*) \, w_g^* \, w_{gh}$ for all $(g,h) \in \cG^{(2)}$.
\item If $(\al,u)$ globally preserves a field of Cartan subalgebras $(A_x \subset B_x)_{x \in X}$ and induces an outer action on the associated equivalence relations, then $w_g$ can be chosen in $\cN_{B_{t(g)}}(A_{t(g)})$.
\end{enumlist}
\end{letterthm}

We also show in Proposition \ref{prop.crossed-product-equiv-rel}  that if $A\subset B \subset M$ is a triple inclusion of von Neumann algebras, with $M$ a II$_1$ factor and $A$ Cartan in $M$, then $B$ is regular in $M$ iff  $A\subset B$ is regular in $M$ (i.e. the normalizer of $A\subset B$ in $M$, $\cN_M(A\subset B) := \{u\in \cU(M)\mid uAu^*=A, uBu^*=B\}$, generates $M$) and iff the corresponding inclusion of equivalence relations $\cR_{A \subset B} \subset \cR_{A \subset M}$ is strongly normal in the sense of \cite{FSZ88}. When combined with the second part of Theorem \ref{thm-A.two-cocycle-vanishing}, this shows that if in addition $M$ is assumed amenable relative to $B$, then such a triple $A \subset B \subset M$ can be identified with $A\subset B \subset B\rtimes_\alpha \cG$, where $\cG=\cG_{B\subset M}$ and $\alpha$ is a genuine action of $\cG$ on $B$
that leaves $A$ invariant. Also the equivalence relation $\cR_{A \subset M}$ can then be written as a genuine semidirect product of $\cR_{A \subset B}$ and an action of $\cG$.
Initiated in \cite{FSZ88}, the study of normal subequivalence relations
saw a revival of interest in recent years, and this structural result about the co-amenable case may be relevant in this direction.

As we mentioned before, while \cite{CFW81} shows the uniqueness up to conjugacy by automorphisms of $R$ of the regular von Neumann subalgebras $B\subset R$ of type I$_n$ satisfying $B'\cap R=\cZ(B)$, the first part of Theorem \ref{thm-A.two-cocycle-vanishing} combined with \cite{Suth-Tak84} shows the uniqueness up to conjugacy by automorphisms of $R$ of regular von Neumann subalgebras $B\subset R$ of type II$_1$ that satisfy $B'\cap R=\cZ(B)$ and have the same groupoid. Since any groupoid $\cG$ has a ``model action'' on $B$ that normalizes a Cartan subalgebra $A$ of $B$ on which it acts freely, this uniqueness result also implies that any regular subalgebra $B\subset R$ contains a Cartan subalgebra of $R$.  We use the second part of Theorem \ref{thm-A.two-cocycle-vanishing} to also prove that any two Cartan subalgebras $A_1, A_2$ of $R$ that are contained in $B$ are conjugate by an automorphism of $R$ that leaves $B$ globally invariant. Altogether, in Theorem \ref{thm.classification-amenable}, Corollary \ref{cor.always-exists-Cartan} and Theorem \ref{thm.classification-amenable-Cartan}, we obtain:

\begin{letterthm}\label{thm-B.classification}
Let $R$ be the hyperfinite II$_1$ factor.
\begin{enumlist}
\item Two regular von Neumann subalgebras $B \subset R$ satisfying $B' \cap R = \cZ(B)$ are conjugate by an automorphism of $R$ if and only if they are of the same type and have isomorphic associated discrete measured groupoids $\cG_{B \subset R}$.
\item Any such $B$ contains a Cartan subalgebra of $R$ and if $A_1,A_2 \subset B$ are Cartan in $R$, there exists an automorphism $\theta$ of $R$ satisfying $\theta(B) = B$ and $\theta(A_1) = A_2$.
\end{enumlist}
\end{letterthm}

Our proofs of Theorems \ref{thm-A.two-cocycle-vanishing} and \ref{thm-B.classification} make use of the known $2$-cocycle vanishing theorems for amenable groups on factors, resp.\ equivalence relations (see \cite{O85,Popa18,FSZ88}), as well as the uniqueness, up to cocycle conjugacy, of free actions of amenable groups on the hyperfinite II$_1$ factor, resp.\ equivalence relation (see \cite{O85,BG84}). We apply these results to the isotropy groups $\Gamma_x = \{g \in \cG \mid s(g) = t(g) = x\}$ of an amenable groupoid $\cG$. In order to extend to the entire groupoid $\cG$, we have to make equivariant choices of $2$-cocycle vanishing, resp.\ cocycle conjugacy, for the $\Gamma_x$, where the equivariance is w.r.t.\ to the isomorphisms $\Gamma_x \recht \Gamma_y$ given by conjugation with an element $g \in \cG$ with $s(g) = x$ and $t(g) = y$.

The proof of this later part depends on two key points. The first one is the technical Theorem \ref{thm.equivariant-section}, saying that such an equivariant choice exists, provided that the $2$-cocycle vanishing, resp.\ cocycle conjugacy for $\Gamma_x$, can be done in an ``approximately unique way''.

More precisely, assume that $(\al,u)$ is an outer cocycle action of the countable amenable group $\Gamma$ on a II$_1$ factor $B$. By \cite{Popa18}, we know that $u$ is a co-boundary: we can choose unitary elements $(w_g)_{g \in \Gamma}$ in $B$ such that $u(g,h) = \al_g(w_h^*) \, w_g^* \, w_{gh}$. Then, $\be_g = \Ad w_g \circ \al_g$ defines an outer action of $\Gamma$ on $B$. Any other choice of $w$ such that $u = \partial w$ is of the form $v_g w_g$, where $(v_g)_{g \in \Gamma}$ is a $1$-cocycle for the action $\be$, meaning that $v_{gh} = v_g \, \be_g(v_h)$. By the above ``approximate uniqueness'' of $w$, we mean that every such $1$-cocycle $v$ is approximately a co-boundary.

The second key point is then to prove such approximate vanishing of $1$-cocycles for arbitrary amenable groups (see Theorem \ref{thm.amenable-1-cocycle-approx-vanish}). We state this result below because of its independent
interest.

\begin{letterthm}\label{thm-C.approximate-vanishing}
Let $\Gamma$ be a countable group. The following properties are equivalent.
\begin{itemlist}[labelwidth=6ex,labelsep=0.5ex,leftmargin=6.5ex]
\item[(i)] $\Gamma$ is amenable.

\item[(ii)] No free trace preserving action $\Gamma \actson^\si (B,\tau)$ on a tracial diffuse von Neumann algebra is strongly ergodic.

\item[(iii)] For every free action $\Gamma \actson^\si N$ on a II$_1$ factor $N$, the fixed point algebra of the ultrapower action $\sigma^\om$ on $N^\om$ is an irreducible subfactor of $N^\om$.

\item[(iv)] Given any free action $\Gamma \actson^\si N$ on a II$_1$ factor $N$, any $1$-cocycle $w=(w_g)_{g \in \Gamma}$ for $\sigma$ is approximately a co-boundary:
for every finite subset $F\subset \Gamma$ and every $\eps >0$, there exists $u\in \cU(N)$ such that $\|u\sigma_g(u^*)-w_g\|_2\leq \varepsilon$ for all $g\in F$.
Equivalently, there exists a unitary $U\in N^\omega$ such that $U\sigma^\omega_g(U^*)=w_g$, $\forall g\in \Gamma$, i.e., $w$ is a co-boundary as a
$1$-cocycle for $\sigma^\omega$.

\end{itemlist}
\end{letterthm}

The proof of the above result is quite subtle, with the statement itself being somewhat surprising, given its generality. Note in this respect that $1$-cocycles for a unitary representation of an amenable group need not be approximately co-boundary (see \cite{Shalom03}). We similarly prove an approximate vanishing of $1$-cocycles with values in the normalizer of a Cartan subalgebra, which we need in the proof of the second part of Theorem \ref{thm-A.two-cocycle-vanishing} (see Theorem \ref{thm.approx-vanish-Cartan}).

In \cite{Popa18}, the class of countable groups satisfying $2$-cohomology vanishing for cocycle actions on II$_1$ factors is denoted as $\VC$. In \cite[Remarks 4.5]{Popa18}, it is speculated that $\VC$ might coincide with the class of treeable groups. We elaborate on this in Section \ref{sec.tree-cost} and this leads us to a new notion of treeability for countable pmp equivalence relations, which we call weak treeability. It is not known whether every treeable group is strongly treeable. We prove that this question is closely related to the question whether every weakly treeable equivalence relation is treeable. Similarly, we introduce a fixed price problem for equivalence relations and relate it to the fixed price problem for groups.

In the final Section \ref{sec.examples-groupoids}, we construct numerous families of amenable discrete measured groupoids. We show that even though there is a unique ergodic amenable equivalence relation of type II$_1$, the classification problem of amenable discrete measured groupoids is strictly harder than the classification of amenable groups. Even restricting to groupoids $\cG$ for which a.e.\ isotropy group is $\Z^2$, one may ``encode'' ergodic transformations of the interval $[0,1]$ up to conjugacy into such amenable groupoids up to isomorphism (see Corollary \ref{cor.classif-vs-non-classif}).

\section{Discrete measured groupoids, cocycle actions and crossed products}\label{sec.groupoids}

In this section, we summarize some of the basic terminology and well known results on discrete measured groupoids and their actions. Recall that a \emph{discrete measured groupoid} $\cG$ is a groupoid with the following extra structure. This concept goes back to \cite{mackey63}.

\begin{itemlist}
\item $\cG$ is a standard Borel space and the units $\cG^{(0)} \subset \cG$ form a Borel subset.
\item The source and target maps $s,t : \cG \recht \cG^{(0)}$ are Borel and countable-to-one.
\item Defining $\cG^{(2)} = \{(g,h) \in \cG \times \cG \mid s(g) = t(h)\}$, the multiplication map $\cG^{(2)} \recht \cG : (g,h) \mapsto gh$ is Borel. The inverse map $\cG \recht \cG : g \mapsto g^{-1}$ is Borel.
\item $\cG^{(0)}$ is equipped with a probability measure $\mu$ such that, denoting by $\mu_s$ and $\mu_t$ the $\si$-finite measures on $\cG$ given by integrating the counting measure over $s,t : \cG \recht \cG^{(0)}$, we have that $\mu_s \sim \mu_t$.
\end{itemlist}

We say that $\cG$ is probability measure preserving (pmp) if $\mu_s = \mu_t$. To every discrete measured groupoid $\cG$ is associated the countable nonsingular equivalence relation $\cR$ on $(\cG^{(0)},\mu)$ given by
$$\cR = \{(t(g),s(g)) \mid g \in \cG \} \; .$$
Note that $\cG$ is pmp if and only if $\cR$ is a pmp equivalence relation. We say that $\cG$ is ergodic if the associated equivalence relation $\cR$ is ergodic.

\begin{remark}
In this paper, we only work in the measurable context, discarding sets of measure zero whenever useful. By von Neumann's measurable selection theorem (see e.g.\ \cite[Theorem 18.1]{kechris-book}), we therefore never have problems choosing measurable sections. Also, all isomorphisms between measure spaces, fields of von Neumann algebras, measured groupoids, are defined up to sets of measure zero. Whenever we write \emph{for all}, this should be interpreted as \emph{for almost every}, whenever appropriate.

Using \cite[Sections 1 and 2]{Sutherland85}, we may consider measurable fields of any kind of `separable structures': von Neumann algebras with separable predual, standard probability spaces, countable groups, Polish spaces, Polish groups, etc. In particular, by \cite[Theorem 2.5]{Sutherland85}, all natural definitions of such measurable fields are equivalent.
\end{remark}

Fix a discrete measured groupoid $\cG$ and write $X = \cG^{(0)}$. Define $\cG^{(2)}$ as above and similarly define $\cG^{(3)}$.

\begin{definition}\label{def.cocycle-action}
A \emph{cocycle action} $(\al,u)$ of $\cG$ on a measurable field $(B_x)_{x \in X}$ of von Neumann algebras with separable predual is given by
\begin{itemlist}
\item a measurable field of $*$-isomorphisms $\cG \ni g \mapsto \al_g : B_{s(g)} \recht B_{t(g)}$,
\item a measurable field of unitaries $\cG^{(2)} \ni (g,h) \mapsto u(g,h) \in \cU(B_{t(g)})$,
\end{itemlist}
satisfying
\begin{align*}
\al_g \circ \al_h = \Ad(u(g,h)) \circ \al_{gh} &\quad\text{for all $(g,h) \in \cG^{(2)}$,}\\
\al_g(u(h,k)) \, u(g,hk)) = u(g,h) \, u(gh,k) &\quad\text{for all $(g,h,k) \in \cG^{(3)}$,}\\
\al_g = \id  &\quad\text{when $g \in \cG^{(0)}$,}\\
u(g,h) = 1 &\quad\text{when $g \in \cG^{(0)}$ or $h \in \cG^{(0)}$.}
\end{align*}
\end{definition}

An automorphism $\al$ of a von Neumann algebra $B$ is said to be free (or properly outer) if the only element $v \in B$ satisfying $v x = \al(x) v$ for all $x \in B$ is the zero element $v=0$. When $B$ is a factor, an automorphism $\al$ is free if and only if it is outer, meaning that there is no unitary element $v \in \cU(B)$ such that $\al = \Ad v$.

Denote by $\Gamma_x = \{g \in \cG \mid s(g) = t(g) = x\}$ the isotropy groups of $\cG$. The cocycle action $(\al,u)$ is said to be free if for almost every $x \in X$ and all $g \in \Gamma_x$ with $g \neq e$, the $*$-automorphism $\al_g : B_x \recht B_x$ is free.

The cocycle actions $(\al,u)$ of $\cG$ on $(B_x)_{x \in X}$ and $(\be,\Psi)$ of $\cG$ on $(D_x)_{x \in X}$ are said to be \emph{cocycle conjugate} if there exists a measurable field of $*$-isomorphisms $X \ni x \mapsto \theta_x : B_x \recht D_x$ and a measurable field of unitaries $\cG \ni g \mapsto w_g \in \cU(D_{t(g)})$ satisfying
\begin{align*}
\theta_{t(g)} \circ \al_g \circ \theta_{s(g)}^{-1} = \Ad w_g \circ \be_g &\quad\text{for all $g \in \cG$,}\\
\theta_{t(g)}(u(g,h)) = w_g \, \be_g(w_h) \, \Psi(g,h) \, w_{gh}^* &\quad\text{for all $(g,h) \in \cG^{(1)}$.}
\end{align*}

The \emph{cocycle crossed product construction} associates to every cocycle action $(\al,u)$ a regular inclusion of von Neumann algebras $B \subset M$ together with a faithful normal conditional expectation $E : M \recht B$. The cocycle crossed product $M$ can be easily defined by first considering the \emph{full pseudogroup} $[[\cG]]$ of $\cG$ consisting of all Borel sets $\cU \subset \cG$ with the property that the restrictions to $\cU$ of the source map $s$ and the target map $t$ are injective, and identifying $\cU,\cU'$ if they differ by a set of measure zero. The composition of $\cU,\cV \in [[\cG]]$ is defined as
$$\cU \cdot \cV = \{g h \mid g \in \cU, h \in \cV , s(g) = t(h) \} \; .$$
For every $\cU \in [[\cG]]$, define the Borel map $\vphi_\cU : s(\cU) \recht t(\cU) : \vphi_\cU(s(g)) = t(g)$ for all $g \in \cU$.

Writing
$$B = \int^{\oplus}_{X} B_x \; d\mu(x) \; ,$$
the cocycle action $(\al,u)$ of $\cG$ on $(B_x)_{x \in X}$ can be reinterpreted as follows in terms of $[[\cG]]$. To every $\cU \in [[\cG]]$, there corresponds a $*$-isomorphism $\al_\cU : B 1_{s(\cU)} \recht B 1_{t(\cU)}$ given by
$$(\al_\cU(b))(t(g)) = \al_g(b(s(g))) \quad\text{for all $b \in B 1_{s(\cU)}$, $g \in \cU$.}$$
To every $\cU,\cV \in [[\cG]]$, there corresponds a unitary element $u(\cU,\cV) \in B 1_{t(\cU \cdot \cV)}$ given by
$$u(\cU,\cV)_{t(gh)} = u(g,h) \quad\text{for all $g \in \cU$ and $h \in \cV$ with $s(g) = t(h)$.}$$
Then, the cocycle crossed product $M$ is generated by $B$ and partial isometries $u(\cU)$ for all $\cU \in [[\cG]]$ satisfying the following properties.
\begin{align*}
& u(\cU)^* u(\cU) = 1_{s(\cU)} \quad\text{and}\quad u(\cU) u(\cU)^* = 1_{t(\cU)} \; ,\\
& u(\cU) \, u(\cV) = u(\cU,\cV) \, u(\cU \cdot \cV) \; ,\\
& u(\cU) b u(\cU)^* = \al_\cU(b) \quad\text{for all $b \in B 1_{s(\cU)}$,}\\
& E(u(\cU)) = 1_{\cU \cap \cG^{(0)}} \; .
\end{align*}
Note that $B' \cap M = L^\infty(X)$ if and only if almost every $B_x$ is a factor and the cocycle action is free.

Conversely, whenever $M$ is a von Neumann algebra with separable predual and $B \subset M$ is a regular von Neumann subalgebra with a faithful normal conditional expectation $E : M \recht B$ and $B' \cap M = \cZ(B)$, there is a canonical discrete measured groupoid $\cG = \cG_{B \subset M}$ and a cocycle action of $\cG$ on the field of factors given by the central decomposition of $B$, such that $B \subset M$ is isomorphic with the cocycle crossed product inclusion. To see this, define $\cP$ as the set of partial isometries $v \in M$ such that the projections $p = v^* v$ and $q = v v^*$ belong to $\cZ(B)$ and $v B v^* = B q$. Define the equivalence relation $\sim$ on $\cP$ by $v \sim w$ if and only if $v \in B w$. Using the usual correspondence between groupoids and inverse semigroups, one identifies $\cP / \mathord{\sim}$ with the full pseudogroup $[[\cG]]$ of an essentially unique discrete measured groupoid $\cG$ with space of units $\cG^{(0)}$ satisfying $L^\infty(\cG^{(0)}) = \cZ(B)$.

Writing $X =\cG^{(0)}$ and writing $B$ as the direct integral of factors $(B_x)_{x \in X}$, the above relation between $\cP$ and the full pseudogroup $[[\cG]]$ provides a cocycle action of $\cG$ on $(B_x)_{x \in X}$. By construction, $B \subset M$ is isomorphic with the cocycle crossed product inclusion.

Two such cocycle crossed product inclusions are isomorphic if and only if the corresponding measured groupoids are isomorphic and their cocycle actions are cocycle conjugate through this isomorphism of groupoids. This bijective correspondence between regular inclusions $B \subset M$ and cocycle crossed products can be viewed in different ways. In \cite{DFP18}, this is interpreted in the language of inverse semigroups, where cocycle actions of groupoids become extensions of inverse semigroups.

Note that the cocycle crossed product of a free cocycle action of $\cG$ on a field $(B_x)_{x \in X}$ of factors admits a faithful normal tracial state if and only if each $B_x$ admits such a faithful normal tracial state and $X = \cG^{(0)}$ admits an invariant probability measure that is equivalent with $\mu$. So, when $(M,\tau)$ is a tracial von Neumann algebra and $B \subset M$ is a regular von Neumann subalgebra satisfying $B' \cap M = \cZ(B)$, we can decompose $M$ as the cocycle crossed product of a cocycle action of a pmp discrete measured groupoid $\cG$ on a field of tracial factors. We get that $M$ is a factor if and only if $\cG$ is ergodic.

Recall from \cite[Section 3.2]{AR00} the notion of amenability for discrete measured groupoids. Also recall from \cite[Section 5.3]{AR00} that $\cG$ is amenable if and only if the countable equivalence relation $\cR$ is amenable and almost all isotropy groups $\Gamma_x$ are amenable.

\section{\boldmath Vanishing of $2$-cohomology for amenable groupoids}\label{sec.vanish-2-cohom}

In this section, we prove the first parts of Theorem \ref{thm-A.two-cocycle-vanishing} and \ref{thm-B.classification}. So we first prove the following result.

\begin{theorem}\label{thm.two-cocycle-vanishing}
Let $\cG$ be a discrete measured groupoid with $X = \cG^{(0)}$ and $(B_x)_{x \in X}$ a measurable field of II$_1$ factors with separable predual. Assume that $\cG$ is amenable.

When $(\al,u)$ is a free cocycle action of $\cG$ on $(B_x)_{x \in X}$, the cocycle $u$ is a co-boundary: there exists a measurable field of unitaries $\cG \ni g \mapsto w_g \in B_{t(g)}$ such that
$$u(g,h) = \al_g(w_h^*) \, w_g^* \, w_{gh} \quad\text{for all $(g,h) \in \cG^{(2)}$.}$$
\end{theorem}

By Theorem \ref{thm.two-cocycle-vanishing}, if $(M,\tau)$ is a tracial von Neumann algebra and $B \subset M$ is a regular von Neumann subalgebra satisfying $B' \cap M = \cZ(B)$ such that the associated groupoid is amenable, then the inclusion $B \subset M$ is isomorphic to a true crossed product inclusion, without $2$-cocycle. So in combination with \cite[Theorem 1.2]{Suth-Tak84}, we obtain the following classification of regular subalgebras of amenable tracial von Neumann algebras.

\begin{theorem}\label{thm.classification-amenable}
Let $R$ be the hyperfinite II$_1$ factor. The regular von Neumann subalgebras $B \subset R$ satisfying $B' \cap M = \cZ(B)$ are completely classified by the associated discrete measured groupoid and the type of $B$.

More precisely, given two such inclusions $B_i \subset R$ with associated groupoids $\cG_i$, $i=1,2$, there exists an automorphism $\theta \in \Aut(R)$ satisfying $\theta(B_1) = B_2$ if and only if $B_1$ and $B_2$ have the same type and there exists a measure class preserving isomorphism $\cG_1 \cong \cG_2$.
\end{theorem}

\begin{remark}\label{rem.what-if-non-factorial}
One may formulate Theorem \ref{thm.classification-amenable} without the factoriality assumption on $R$, i.e.\ for arbitrary amenable tracial von Neumann algebras $(M,\tau)$ and their regular von Neumann subalgebras $B \subset M$ satisfying $B' \cap M = \cZ(B)$. In that case, denote for every $n \geq 1$, by $z_n \in \cZ(B)$ the largest central projection such that $B z_n$ is of type I$_n$. Let $z_0 = 1-\sum_{n=1}^\infty z_n$, so that $B z_0$ is of type II$_1$. Since $B \subset M$ is regular, we have that $z_n \in \cZ(M)$ for every $n \geq 0$. The discrete measured groupoid $\cG$ associated with $B \subset M$ can then be viewed as the disjoint union of the groupoids $\cG^{(n)} = \cG_{B z_n \subset M z_n}$.

Note that since nontrivial groups do not have free actions on type I factors, each of the groupoids $\cG^{(n)}$, $n \geq 1$, has trivial isotropy groups, meaning that each $\cG^{(n)}$ is a countable equivalence relation.

The general statement of Theorem \ref{thm.classification-amenable} then goes as follows. Given, for $i=1,2$, amenable tracial von Neumann algebras $(M_i,\tau_i)$ with regular von Neumann subalgebras $B_i \subset M_i$ satisfying $B_i' \cap M_i = \cZ(B_i)$, denote by $\cG_i^{(n)}$ the associated discrete measured groupoids. Then there exists a $*$-isomorphism $\theta : M_1 \recht M_2$ satisfying $\theta(B_1) = B_2$ if and only if for every $n \geq 0$, there is a measure class preserving isomorphism $\cG_1^{(n)} \cong \cG_2^{(n)}$.
\end{remark}

\begin{proof}[Proof of Theorem \ref{thm.two-cocycle-vanishing}]
Denote by $\Gamma_x = \{g \in \cG \mid t(g) = s(g) = x\}$ the field of isotropy groups. Denote by $B \subset M$ the cocycle crossed product and define $D = \cZ(B)' \cap M$. Note that $B \subset D \subset M$ and $\cZ(B) = \cZ(D)$. By construction, we have a measurable field of cocycle actions $\Gamma_x \actson B_x$ and $D$ is the direct integral of the cocycle crossed products $B_x \rtimes \Gamma_x$. By \cite[Theorem 0.1]{Popa18}, every cocycle action of an amenable group on a II$_1$ factor can be perturbed to a genuine action. So, after a perturbation of $(\al,u)$, we may assume that $D_x$ is the crossed product of $B_x$ and a genuine action of $\Gamma_x$, which means that $u(g,h)=1$ whenever $s(g) = t(g)$. We denote by $u(x,g)$, $x \in X$, $g \in \Gamma_x$, the canonical unitary elements realizing this crossed product decomposition $D_x =B_x \rtimes \Gamma_x$.

Denote by $\cR$ the countable, nonsingular equivalence relation on $(X,\mu)$ given by
$$\cR = \{(t(g),s(g)) \mid g \in \cG\} \; .$$
Since $\cG$ is amenable, $\cR$ is an amenable equivalence relation. So, up to measure zero, $X$ can be partitioned into $\cR$-invariant Borel subsets $X = X_0 \sqcup X_1$ such that the restriction of $\cR$ to $X_0$ has finite orbits and thus admits a fundamental domain, while the restriction of $\cR$ to $X_1$ is implemented by a free nonsingular action of the group $\Z$. It follows that the groupoid morphism $\cG \recht \cR : g \mapsto (t(g),s(g))$ admits a measurable lift $q : \cR \recht \cG$ that is itself a morphism, meaning that $q(x,y) q(y,z) = q(x,z)$ for all $(x,y),(y,z) \in \cR$.

We may then view $\cG$ as the semidirect product of the field $(\Gamma_x)_{x \in X}$ of groups and the measurable family of group isomorphisms $\delta_{(x,y)} : \Gamma_y \recht \Gamma_x$ given by $\delta_{(x,y)}(g) = q(x,y) g q(x,y)^{-1}$ for all $(x,y) \in \cR$ and all $g \in \Gamma_y$. Note that $\delta_{(x,y)} \circ \delta_{(y,z)} = \delta_{(x,z)}$.

For the same reason as above, we can perturb $(\al,u)$ so that the composition with $q$ is a genuine action of $\cR$ on the field $(B_x)_{x \in X}$. This means that we have written $M$ as the crossed product of
$$D = \int^{\oplus}_X (B_x \rtimes \Gamma_x) \; d\mu(x)$$
by an action of $\cR$ given by a field of $*$-isomorphisms $\theta_{(x,y)} : (B_y \subset B_y \rtimes \Gamma_y) \recht (B_x \subset B_x \rtimes \Gamma_x)$ of the form
\begin{align*}
\theta_{(x,y)}(b) = \al_{q(x,y)}(b) &\quad\text{for all $(x,y) \in \cR$ and $b \in B_y$,}\\
\theta_{(x,y)}( u(y,g)) \in \cU(B_x) \, u(x,\delta_{(x,y)}(g)) &\quad\text{for all $(x,y) \in \cR$ and $g \in \Gamma_y$,}
\end{align*}
and satisfying $\theta_{(x,y)} \circ \theta_{(y,z)} = \theta_{(x,z)}$.

In order to prove the theorem, we have to show that there exist measurable families of unitaries $v(x,y) \in \cU(B_x)$ (for $(x,y) \in \cR$) and $w(y,g) \in \cU(B_y)$ (for $y \in X$, $g \in \Gamma_y$) such that
\begin{itemlist}
\item $v$ is a $1$-cocycle for the action $\theta$ of $\cR$ on $(B_x)_{x \in X}$, in the sense that $v(x,y) \, \al_{(x,y)}(v(y,z)) = v(x,z)$ for all $(x,y),(y,z) \in \cR$,
\item for all $y \in X$, we have that $(w(y,g))_{g \in \Gamma_y}$ is a $1$-cocycle for the action $\Gamma_y \actson B_y$,
\item we have $(\Ad(v(x,y)) \circ \theta_{(x,y)}) \, \bigl( \, w(y,g) u(y,g) \, \bigr) = w(x,\delta_{(x,y)}(g)) u(x,\delta_{(x,y)}(g))$ for all $(x,y) \in \cR$ and $g \in \Gamma_y$.
\end{itemlist}
Indeed, once this is proved, we can perturb $(\al,u)$ in such a way that $D_x$ is still the crossed product of $B_x$ and a genuine action of $\Gamma_x$ (with implementing unitary elements $w(x,g) u(x,g)$, $g \in \Gamma_x$) and such that $M$ is the crossed product by an action of $\cR$ on $D$ given by genuine conjugacies between these actions. This means that we have decomposed $M$ as the crossed product of $B$ and a genuine action of the groupoid $\cG$, which is exactly what we have to prove.

So it remains to prove the statement above. Denote by $\cP_x$ the normalizer of $B_x$ inside $D_x$. Note that $\cP_x \subset \cU(D_x)$ is a closed subgroup and that $\cP_x = \{ b u(x,g) \mid b \in \cU(B_x), g \in \Gamma_x\}$. We obtain the natural homomorphism $\cP_x \recht \Gamma_x : b u(x,g) \mapsto g$. We denote by $P_x$ the Polish space of homomorphic lifts $\Gamma_x \recht \cP_x$, which we can identify with the space of $1$-cocycles for the action $\Gamma_x \actson B_x$. We have a natural action $\cU(B_x) \actson P_x$ given by conjugation.

Given a lifting homomorphism $\gamma : \Gamma_y \recht \cP_y$ and given $(x,y) \in \cR$, we define the lifting homomorphism $\beta_{(x,y)}(\gamma) : \Gamma_x \recht \cP_x$ given by
$$\beta_{(x,y)}(\gamma) = \theta_{(x,y)} \circ \gamma \circ \delta_{(x,y)}^{-1} \; .$$
We have found a measurable field of homeomorphisms $\beta_{(x,y)} : P_y \recht P_x$ satisfying $\beta_{(x,y)} \circ \beta_{(y,z)} = \beta_{(x,z)}$. We also have the measurable field of Polish group isomorphisms $\theta_{(x,y)} : \cU(B_y) \recht \cU(B_x)$ satisfying $\theta_{(x,y)} \circ \theta_{(y,z)} = \theta_{(x,z)}$. By construction,
$$\beta_{(x,y)}( b \cdot \gamma) = \theta_{(x,y)}(b) \cdot \beta_{(x,y)}(\gamma)$$
meaning that $\theta_{(x,y)}, \be_{(x,y)}$ form a measurable field of conjugacies between the continuous Polish group actions $\cU(B_x) \actson P_x$.

By Theorem \ref{thm.amenable-1-cocycle-approx-vanish} below, the actions $\cU(B_x) \actson P_x$ have dense orbits. By Theorem \ref{thm.equivariant-section} below, we can thus find a measurable section $X \ni x \mapsto \pi(x) \in P_x$ and a $1$-cocycle $\cR \ni (x,y) \mapsto v(x,y) \in \cU(B_x)$ satisfying
$$v(x,y) \cdot \be_{(x,y)}(\pi(y)) = \pi(x) \quad\text{for all $(x,y) \in \cR$.}$$
Writing $\pi(x)(g) = w(x,g) u(x,g)$, the theorem is proved.
\end{proof}

\begin{corollary}\label{cor.always-exists-Cartan}
Let $(M,\tau)$ be an amenable tracial von Neumann algebra with separable predual and $B \subset M$ a regular von Neumann subalgebra. Then $B' \cap M = \cZ(B)$ if and only if there exists a Cartan subalgebra $A \subset M$ with $A \subset B$.
\end{corollary}
\begin{proof}
If $A \subset M$ is a Cartan subalgebra with $A \subset B$, we have $B' \cap M \subset A' \cap M = A \subset B$. Therefore, $B' \cap M = \cZ(B)$.

Conversely, assume that $B \subset M$ is a regular von Neumann subalgebra satisfying $B' \cap M = \cZ(B)$. As in Remark \ref{rem.what-if-non-factorial}, denote for every $n \geq 1$, by $z_n \in \cZ(B)$ the largest central projection such that $B z_n$ is of type I$_n$. Since $B \subset M$ is regular, we have that $z_n \in \cZ(M)$. Since it is straightforward to construct Cartan subalgebras $A_n \subset M z_n$ with $A_n \subset B z_n$ (and such that $B z_n \cong \M_n(\C) \ot A_n$), we may assume that $B$ is of type II$_1$.

By Theorem \ref{thm.two-cocycle-vanishing}, we can write $B \subset M$ as the crossed product inclusion $B \subset B \rtimes \cG$, where $\cG$ is a pmp amenable discrete measured groupoid acting freely on $B$. Write $X = \cG^{(0)}$ and choose any free action of $\cG$ on a field of Cartan subalgebras $(C_x \subset D_x)_{x \in X}$, with each $D_x$ being the hyperfinite II$_1$ factor. Let $C \subset D \subset N$ be the corresponding crossed product inclusion. By Theorem \ref{thm.classification-amenable}, there exists a $*$-isomorphism $\theta : N \recht M$ satisfying $\theta(D) = B$. Writing $A = \theta(C)$, we have found the required Cartan subalgebra of $M$.
\end{proof}

Although the formulation of the following result looks quite different from the `cohomology lemmas' in \cite[Appendix]{Jones-Takesaki-82} and \cite[Theorem 5.5]{Sutherland85}, the proof is very similar to the proof of \cite[Theorem 5.5]{Sutherland85}.

\begin{theorem}\label{thm.equivariant-section}
Let $\cR$ be a countable nonsingular equivalence relation on the standard probability space $(X,\mu)$. Let $(G_x \actson P_x)_{x \in X}$ be a measurable field of continuous actions of Polish groups on Polish spaces, on which $\cR$ is acting by conjugacies: we have measurable fields of group isomorphisms $\cR \ni (x,y) \mapsto \delta_{(x,y)} : G_y \recht G_x$ and homeomorphisms $\cR \ni (x,y) \mapsto \be_{(x,y)} : P_y \recht P_x$ satisfying
$$\delta_{(x,y)} \circ \delta_{(y,z)} = \delta_{(x,z)} \quad , \quad \be_{(x,y)} \circ \be_{(y,z)} = \be_{(x,z)} \quad\text{and}\quad \beta_{(x,y)}(g \cdot \pi) = \delta_{(x,y)}(g) \cdot \beta_{(x,y)}(\pi)$$
for all $(x,y),(y,z) \in \cR$ and $g \in G_y$, $\pi \in P_y$. Let $X \ni x \mapsto \pi(x) \in P_x$ be a measurable section.

Assume that $\cR$ is amenable and assume that for all $(x,y) \in \cR$, the element $\pi(x)$ belongs to the closure of $G_x \cdot \be_{(x,y)}(\pi(y))$.

Then, there exists a measurable family $\cR \ni (x,y) \mapsto v(x,y) \in G_x$ and a section $X \ni x \mapsto \pi'(x) \in P_x$ satisfying the following properties.
\begin{itemlist}
\item $v$ is a $1$-cocycle: $v(x,y) \, \delta_{(x,y)}(v(y,z)) = v(x,z)$ for all $(x,y), (y,z) \in \cR$.
\item $v(x,y) \cdot \be_{(x,y)}(\pi'(y)) = \pi'(x)$ for all $(x,y) \in \cR$.
\end{itemlist}
\end{theorem}

\begin{proof}
By \cite[Theorem 10]{CFW81}, the equivalence relation $\cR$ is hyperfinite. Fix subequivalence relations $\cR_0 \subset \cR_1 \subset \cdots$ having finite orbits such that $\cR = \bigcup_n \cR_n$, up to measure zero. We assume that $\cR_0 = \{(x,x) \mid x \in X\}$ is the trivial subequivalence relation. To prove the theorem, we inductively define $1$-cocycles $\cR_n \ni (x,y) \mapsto v_n(x,y) \in G_x$ and sections $X \ni x \mapsto \pi_n(x) \in P_x$ with the following properties.
\begin{itemlist}
\item For every $n \geq 0$, the restriction of $v_{n+1}$ to $\cR_n$ equals $v_n$.
\item Defining $\beta_{n,(x,y)}(\pi) = v_n(x,y) \cdot \beta_{(x,y)}(\pi)$ for all $\pi \in P_y$, $(x,y) \in \cR_n$, $n \geq 0$, we have $\beta_{n,(x,y)}(\pi_n(y)) = \pi_n(x)$ for all $(x,y) \in \cR_n$.
\item For all $n \geq 0$ and all $x \in X$, we have that $\pi_n(x) \in G_x \cdot \pi(x)$.
\item For all $x \in X$, the sequence $\pi_n(x) \in P_x$ is convergent.
\end{itemlist}
Once these statements are proven, we can define $\pi'(x) = \lim_n \pi_n(x)$ and $v(x,y) = v_n(x,y)$ whenever $(x,y) \in \cR_n$.

Choose a measurable family of metrics $d_x$ on $P_x$ that induce the topology on $P_x$. Define $v_0(x,x) = 1$ and $\pi_0(x) = \pi(x)$. Then assume that $n \geq 0$ and that $v_n$ and $\pi_n$ have been defined. Write $\beta_{n,(x,y)}(\pi) = v_n(x,y) \cdot \beta_{(x,y)}(\pi)$. Choose a fundamental domain $X_0 \subset X$ for $\cR_n$. Every $y \in X_0$ has a finite $\cR_n$-orbit. Therefore, since $\beta_{n,(x,y)} : P_y \recht P_x$ is a homeomorphism, we can choose a measurable function $y \mapsto \eps_y > 0$ such that
\begin{multline*}
d_x(\beta_{n,(x,y)}(\pi),\beta_{n,(x,y)}(\pi_n(y))) < 2^{-n} \\ \text{whenever $y \in X_0, (x,y) \in \cR_n, \pi \in P_y$ and $d_y(\pi,\pi_n(y)) < \eps_y$.}
\end{multline*}
Denote by $\cS$ the restriction of $\cR_{n+1}$ to $X_0$. Choose a fundamental domain $X_1 \subset X_0$ for $\cS$. Note that $X_1$ also is a fundamental domain for $\cR_{n+1}$.

Whenever $y \in X_1$ and $(x,y) \in \cS$, by assumption, $\pi(x)$ belongs to the closure of
$$G_x \cdot \beta_{(x,y)}(\pi(y)) = G_x \cdot \beta_{(x,y)}(\pi_n(y))  \; .$$
Since $\pi_n(x) \in G_x \cdot \pi(x)$, also $\pi_n(x)$ belongs to the closure of $G_x \cdot \beta_{(x,y)}(\pi_n(y))$.

Defining $\cS_1 = \{(x,y) \in \cS \mid x \in X_0 \setminus X_1 , y \in X_1 \}$, we can thus choose a measurable map $\cS_1 \ni (x,y) \mapsto w(x,y) \in G_x$ such that
$$d_x(\pi_n(x),w(x,y) \cdot \beta_{(x,y)}(\pi_n(y))) < \min \{\eps_x, 2^{-n}\} \quad\text{for all $(x,y) \in \cS_1$.}$$
Since $\cR_{n+1}$ is the free product of $\cR_n$ and $\cS$, we can define $v_{n+1}$ as the unique $1$-cocycle $\cR_{n+1} \ni (x,y) \mapsto v_{n+1}(x,y) \in G_x$ satisfying $v_{n+1}(x,y) = v_n(x,y)$ for all $(x,y) \in \cR_n$ and $v_{n+1}(x,y) = w(x,y)$ for all $(x,y) \in \cS$. Define
$$\beta_{n+1,(x,y)}(\pi) = v_{n+1}(x,y) \cdot \beta_{(x,y)}(\pi) \quad\text{for all $(x,y) \in \cR_{n+1}$ and $\pi \in P_y$.}$$

Uniquely define the section $X \ni x \mapsto \pi_{n+1}(x) \in P_x$ given by
\begin{align*}
& \pi_{n+1}(x) = \pi_n(x) \quad\text{if $x \in X_1$, and}\\
& \pi_{n+1}(x) = \beta_{n+1,(x,y)}(\pi_n(y)) \quad\text{if $x \in X \setminus X_1$, $y \in X_1$, $(x,y) \in \cR_{n+1}$.}
\end{align*}
By construction, $\beta_{n+1,(x,y)}(\pi_{n+1}(y)) = \pi_{n+1}(x)$ for all $(x,y) \in \cR_{n+1}$.

When $x \in X_0 \setminus X_1$, take the unique $y \in X_1$ with $(x,y) \in \cR_{n+1}$. Then $(x,y) \in \cS_1$, so that $\pi_{n+1}(x) = w(x,y) \cdot \beta_{(x,y)}(\pi_n(y))$ and thus
$$d_x(\pi_{n+1}(x),\pi_n(x)) < \min \{\eps_x, 2^{-n}\} \; .$$
When $x \in X \setminus X_0$, take the unique $y \in X_0$ with $(x,y) \in \cR_n$. Since $d_y(\pi_{n+1}(y),\pi_n(y)) < \eps_y$ and $\beta_{n,(x,y)}(\pi_n(y)) = \pi_n(x)$, we find that
$$d_x(\beta_{n,(x,y)}(\pi_{n+1}(y)),\pi_n(x)) < 2^{-n} \; .$$
Because $(x,y) \in \cR_n$, we get that $\beta_{n,(x,y)}(\pi_{n+1}(y)) = \beta_{n+1,(x,y)}(\pi_{n+1}(y)) = \pi_{n+1}(x)$. So, we have proved that $d_x(\pi_{n+1}(x),\pi_n(x)) < 2^{-n}$ for all $x \in X$.

Continuing inductively, it follows that for all $x \in X$, the sequence $\pi_n(x) \in P_x$ is convergent. This concludes the proof of the theorem.
\end{proof}

\section{\boldmath Approximate vanishing of $1$-cocycles}

Recall that a trace preserving action $\Gamma \actson^\si (B,\tau)$ on a tracial von Neumann algebra is called \emph{strongly ergodic} if any sequence $b_n$ in the unit ball of $B$ satisfying the approximate invariance property $\lim_n \|\si_g(b_n)-b_n\|_2 = 0$ for all $g \in \Gamma$ must be approximately scalar, in the sense that $\lim_n \|b_n - \tau(b_n)1\|_2 = 0$.

\begin{theorem}\label{thm.amenable-1-cocycle-approx-vanish}
Let $\Gamma$ be a countable group. The following properties are equivalent.
\begin{itemlist}[labelwidth=6ex,labelsep=0.5ex,leftmargin=6.5ex]
\item[(i)] $\Gamma$ is amenable.

\item[(ii)] No free trace preserving action $\Gamma \actson^\si (B,\tau)$ on a tracial diffuse von Neumann algebra is strongly ergodic.

\item[(ii')] For every free trace preserving action $\Gamma \actson^\si (B,\tau)$ on a tracial diffuse von Neumann algebra and for every non-principal ultrafilter $\om$ on $\N$, the ultrapower action $\si^\om$ on $\cB = B^\om$ has a diffuse fixed point algebra $\cB^{\si^\om}$.

\item[(ii'')] The Bernoulli $\Gamma$-action $\Gamma\curvearrowright^\sigma R= (M_{2\times 2}(\C))^{\ovt \Gamma}$ with base $M_{2\times 2}(\C)$ is not strongly ergodic.

\item[(iii)] For every free action $\Gamma \actson^\si N$ on a II$_1$ factor $N$, every $1$-cocycle $(w_g)_{g \in \Gamma}$ is approximately a co-boundary: for every finite subset $F\subset \Gamma$ and every $\eps >0$, there exists $u\in \cU(N)$ such that $\|u\sigma_g(u^*)-w_g\|_2\leq \varepsilon$ for all $g\in F$.

\item[(iii')] The previous statement with $N$ being the hyperfinite II$_1$ factor.

\item[(iv)] Let $\Gamma\curvearrowright^{\sigma_n} N_n$ be a sequence of free $\Gamma$-actions on II$_1$ factors. Let $\omega$ be a non-principal ultrafilter on $\N$ and denote by $\cN=\prod_\omega N_n$ the corresponding ultraproduct II$_1$ factor with $\sigma=\prod_\omega \sigma_n$ the ultraproduct $\Gamma$-action, where $\sigma_g((x_n)_n)=(\sigma_{n,g}(x_n))_n$. Then, the fixed point algebra $\cN^\sigma$ is an irreducible subfactor of $\cN$.

\item[(v)] In the same setting as {\it (iv)}, every $1$-cocycle for $\sigma=\prod_\omega \sigma_n$ is a co-boundary.
\end{itemlist}
\end{theorem}

\begin{proof}
{\it (i) $\Rightarrow$ (iv).} It is sufficient to prove that given any $x\in (\cN)_1$, there exists a
von Neumann subalgebra $B\subset \cN^\sigma$ such that $E_{B'\cap \cN}(x)=\tau(x) 1$. Let $x=(x_n)_n$ with $x_n\in (N_n)_1$ and let $F_n\subset \Gamma$ be an increasing sequence of finite subsets exhausting $\Gamma$. By Lemma \ref{lem.three}, for each $n$, there exists a finite dimensional abelian von Neumann subalgebra $A_n\subset N$ such that
\begin{align}
& \|\sigma_g(y)-y\|_2\leq 2^{-n} \quad\text{for all}\;\; y\in (A_n)_1, g\in F_n \; , \label{eq.4.1} \\
& \|E_{A_n'\cap N}(x_n)-\tau(x_n)1\|_2\leq 2^{-n} \; . \label{eq.4.2}
\end{align}
By \eqref{eq.4.1}, it follows that $B=\prod_\omega A_n$ is contained in the fixed point algebra $\cN^\si$, while \eqref{eq.4.2} implies that $E_{B'\cap \cN}(x)=\tau(x)1$.

{\it (iv) $\Rightarrow$ (iii).} We use Connes' $2$-by-$2$ matrix trick, as in the proof of \cite[Proposition 1.3.1$^\circ$]{Popa18}. Denote be $\tilde{\sigma}$ the $\Gamma$-action on $\tilde{N}=M_{2 \times 2}(N)=N \otimes M_{2 \times 2}(\C)$ given by $\tilde{\sigma}_g=\sigma_g \otimes \id$. Denoting by $\{e_{ij} \mid 1\leq i,j \leq 2\}$ the matrix units for $M_{2 \times 2}(\C)\subset \tilde{N}$, we get that $\tilde{w}_g=e_{11}+ w_g e_{22}$ is a $1$-cocycle for $\tilde{\sigma}$, and thus also for the ultrapower action $\tilde{\sigma}^\omega$ on ${\tilde{N}}^\om$.

Note that if we denote $\sigma'_g = (\Ad \tilde{w}_g) \circ \tilde{\sigma}_g$, then its ultrapower $(\sigma')^\omega$ coincides with $(\Ad \tilde{w}_g) \circ \tilde{\sigma}^\omega$. Thus, since we assume that {\it (iv)} holds, the fixed point algebra $Q$ of the $\Gamma$-action $(\si')^\om$ on $\tilde{N}^\omega$ is an irreducible subfactor of $\tilde{N}^\om$. Since $e_{11}, e_{22}$ are projections in $Q$, it follows that they are equivalent in $Q$ via a partial isometry $u^* e_{12}\in Q$, for some $u\in \cU(N^\omega)$. But the fact that $u^* e_{12}$ lies in the fixed point algebra $Q$ is equivalent to the fact that
$w_g = u \sigma_g^\omega(u^*)$ for all $g\in \Gamma$, i.e., $w$ is a co-boundary for $\sigma^\omega$. Then, {\it (iii)} follows.

{\it (iii) $\Rightarrow$ (v).} By {\it (iii)}, every $1$-cocycle for $\si$ is approximately a co-boundary. By the usual re-indexation argument, the $1$-cocycle is a true co-boundary.

{\it (v) $\Rightarrow$ (iii)} follows by taking $N_n = N$ for all $n$.

{\it (ii') $\Rightarrow$ (ii) $\Rightarrow$ (ii'')}, as well as {\it (iii) $\Rightarrow$ (iii')}, are trivial.

{\it (iv) $\Rightarrow$ (ii').} Let $z \in \cZ(B)$ be the largest projection such that $\cZ(B) z$ is diffuse. Then, $z$ is $\Gamma$-invariant and it follows from \cite{S81} that $\cZ(\cB)^{\si^\om} z$ is diffuse. For the atomic part, it follows from {\it (iv)} that $\cB^{\si^\om} (1-z)$ is diffuse.

The implication {\it (ii'') $\Rightarrow$ (i)} is contained in \cite{J81}.

{\it (iii') $\Rightarrow$ (i).} For this, we follow a similar argument to \cite[Proposition 1.3.2$^\circ$]{Popa18} and \cite[Theorem 3.2]{P01}. Let $(N_0, \varphi_0)$ be a copy of the $2$ by $2$ matrix algebra with the state given by the weights $1/3$ and $2/3$. Let $\Gamma \curvearrowright^\al (\cN, \varphi)=\overline{\otimes}_g (N_0, \varphi_0)_g$ be the Bernoulli $\Gamma$-action with base $(N_0, \varphi_0)$ and denote by $N=\cN_\vphi$ the centralizer of the state $\vphi$. Note that $N$ is isomorphic with the hyperfinite II$_1$ factor. Denote by $\si$ the corresponding Connes-St{\o}rmer Bernoulli action, given by restricting $\al$ to the II$_1$ factor $N$.

Let $B \subset N$ be a copy of the $2$ by $2$ matrix algebra with matrix units $(e_{ij})_{1 \leq i,j \leq 2}$. By our choice of $\vphi_0$ and because $N$ is a factor, we can choose an isometry $v \in \cN$ such that $\si_t^\vphi(v) = 2^{-it} v$ and $vv^* = e_{11}$. As in \cite[Proposition 1.3.2$^\circ$]{Popa18} and \cite[Theorem 3.2]{P01}, define the $1$-cocycle $(w_g)_{g \in \Gamma}$ for $\sigma$ given by
$$w_g = v \al_g(v^*) + e_{21} v \al_g(v^* e_{12}) \; .$$
Since {\it (iii')} holds, $w$ is approximately a co-boundary. Take a sequence of unitaries $a_n \in \cU(N)$ such that $\lim_n \|w_g \si_g(a_n) - a_n\|_2 = 0$ for all $g \in \Gamma$. Working in $L^2(\cN,\vphi)$, it follows that $\lim_n \|v^* w_g \si_g(a_n) - v^* a_n \|_2 = 0$ for all $g \in \Gamma$. But, $v^* w_g = \al_g(v^*)$ and we conclude that $\lim_n \|\al_g(v^* a_n) - v^* a_n \|_2 = 0$ for all $g \in \Gamma$. Since $\si_t^\vphi(v^* a_n) = 2^{it} v^* a_n$, the sequence $v^* a_n$ lies in $L^2(\cN,\vphi) \ominus \C 1$. Since
$$\|v^* a_n\|_2 = \|vv^* a_n\|_2 = \|e_{11} a_n\|_2 = \frac{1}{\sqrt{2}} \; ,$$
we conclude that the unitary representation $\pi$ of $\Gamma$ on $L^2(\cN,\vphi) \ominus \C 1$ given by the action $\al$ weakly contains the trivial representation of $\Gamma$. As in \cite{J81}, $\pi$ is weakly contained in the regular representation of $\Gamma$. So it follows that $\Gamma$ is amenable.
\end{proof}

\begin{lemma}\label{lem.one}
Let $\Gamma \curvearrowright^\sigma N$ be a free action of a countable group $\Gamma$ on a separable II$_1$ factor $N$.
Let $\omega$ be a non-principal ultrafilter on $\N$ and $\sigma^\omega$ the ultrapower action of $\Gamma$ on $N^\omega$ defined by $\sigma_g((x_n)_n)=(\sigma_g(x_n))_n$ for all $(x_n)_n \in N^\omega$ and $g\in \Gamma$.

There exists a unitary element $u\in N^\omega$ such that  $N, \sigma^\omega_g(uNu^*), g\in \Gamma$, are freely independent. Thus, if one denotes $\tilde{N}=\bigvee_g \sigma^\omega_g(uNu^*)$ then $\sigma^\omega$ leaves $\tilde{N}$ invariant and the restriction of $\sigma^\omega$ to $N\vee \tilde{N}$ is isomorphic to the diagonal free product $\Gamma \curvearrowright N* N^{*\Gamma}$ between $\sigma$ and the free Bernoulli $\Gamma$-action with base $N$.
\end{lemma}

\begin{proof}
Since $N$ has trivial relative commutant in $M=N\rtimes \Gamma$, by the main theorem of \cite{P92} and \cite[Theorem 0.1]{P13}, there exists a Haar unitary $u\in N^\omega$
such that $u$ is freely independent to the separable II$_1$ factor $M$. But then, if we denote by $(u_g)_g \in M$ the canonical unitaries implementing the action $\sigma$,
it is easy to see that $N, u_guNu^*u_g^*$, $g\in \Gamma$, are all mutually free von Neumann subalgebras. Since $\sigma$ leaves $\bigvee_g u_guNu^*u_g^*$
invariant, and implements on it the free Bernoulli $\Gamma$-action, the lemma follows.
\end{proof}

\begin{lemma}\label{lem.two}
Let $\Gamma$ be a countable amenable group, $N$ a II$_1$ factor and $\Gamma \curvearrowright^\rho N^{*\Gamma}$ the free Bernoulli $\Gamma$-action with base $N$. Then $\rho$ is not strongly ergodic.
\end{lemma}
\begin{proof}
Let $a=a^*\in N$ be a semi-circular element and denote by $a_g$ its identical copies in position $g \in \Gamma$ inside the free product $N^{*\Gamma}$. Thus, $\rho$ acts on the set $\{a_g\}_g$ by left translation: $\rho_h(a_g)=a_{hg}$. Let $K_n\subset \Gamma$ be a
sequence of F{\o}lner sets and denote $b_n=|K_n|^{-1/2}\sum_{g\in K_n} a_g$. Then $b_n$ is also a semicircular element
and one has
$$
\|\rho_h(b_n)-b_n\|_2^2=|hK_n \Delta K_n|/|K_n| \rightarrow 0 \quad\text{for all}\;\; h\in \Gamma \; .
$$
Thus, the element $\tilde{b}=(b_n)_n \in (N^{*\Gamma})^\omega$ is semicircular with $\rho_h(\tilde{b})=\tilde{b}$ for all $h\in \Gamma$, showing that $\rho$ is not strongly ergodic.
\end{proof}

Note that the construction of $b_n$ in the proof of Lemma \ref{lem.two} can also be interpreted as follows. We apply the free Gaussian functor to the regular representation of $\Gamma$. Since $\Gamma$ is amenable, the regular representation contains a sequence of almost invariant unit vectors. This provides approximately $\Gamma$-invariant semicircular elements, i.e.\ the elements $b_n$.

\begin{lemma}\label{lem.three}
Let $\Gamma \curvearrowright^\sigma N$ be a free action of a countable amenable group $\Gamma$ on a II$_1$ factor $N$. Given any finite subsets $F_0 \subset \Gamma, X_0 \subset N$ and any $\delta > 0$, there exists a $d$-dimensional abelian von Neumann subalgebra $A_0\subset N$ with all minimal projections of trace $1/d$, such that
$$\|\sigma_g(a)-a\|_2\leq \delta \quad\text{and}\quad \|E_{A_0' \cap N}(x) - \tau(x)1\|_2 \leq \delta$$
for all $a\in (A_0)_1$, $g\in F_0$ and $x \in X_0$.
\end{lemma}
\begin{proof}
By \cite{P13}, there exists a $\sigma$-invariant separable II$_1$ subfactor $N_0\subset N$ that contains $X_0$ and on which $\Gamma$ acts freely.
By Lemmas \ref{lem.one} and \ref{lem.two}, $N_0^\omega \subset N^\omega$ contains a $\sigma^\omega$-invariant II$_1$ subfactor $P\subset N^\omega$ that is freely independent to $N_0$ and on which $\rho={\sigma^\omega}|_P$ is free and non strongly ergodic.

Since $P$ is freely independent to $N_0$ and $X_0 \subset N_0$, we can fix $d \geq 1$ such that $\|E_{B_0' \cap N^\om}(x) - \tau(x)1\|_2 \leq \delta/2$ for every $x \in X_0$ and for every $d$-dimensional abelian von Neumann subalgebra $B_0\subset P$ with all minimal projections having trace $1/d$.

Since the action $\rho$ of $\Gamma$ on $P$ is not strongly ergodic, we can make such a choice of $B_0 \subset P$ such that $\|\sigma_g(b)-b\|_2\leq \delta/2$ for all $b \in (B_0)_1$ and $g \in F_0$. Writing $B_0$ as an ultraproduct of $d$-dimensional abelian von Neumann subalgebras $B_{0,n}\subset N$ with all minimal projections having trace $1/d$, the subalgebra $A_0 = B_{0,n}$ satisfies all required conditions when $n$ is large enough.
\end{proof}

In \cite{Popa18}, $\VC$ is defined as the class of countable groups $\Gamma$ with the property that for every free cocycle action $(\al,u)$ of $\Gamma$ on a II$_1$ factor $B$, the $2$-cocycle $u$ is a co-boundary. It is proved in \cite{Popa18} that the class $\VC$ is closed under amalgamated free products over finite subgroups, but is not closed under amalgamated free products over amenable subgroups. One similarly defines the class $\VC^\omega$ of countable groups for which every such $2$-cocycle $u$ is approximately a co-boundary. As a corollary to Theorem \ref{thm.amenable-1-cocycle-approx-vanish}, we prove that $\VC^\omega$ is closed under amalgamated free products over amenable subgroups.

\begin{corollary}
The class $\VC^\omega$ is closed under amalgamated free products over amenable subgroups. In particular, if $\Gamma_1, \Gamma_2$ are free products of amenable groups and $H\subset \Gamma_1, \Gamma_2$ is an amenable subgroup, then $\Gamma_1*_H \Gamma_2\in \VC^\omega$.
\end{corollary}

\begin{proof}
The same proof as the one of \cite[Proposition 1.5]{Popa18} shows that if $\Gamma_1, \Gamma_2\in \VC^\omega$ and $H\subset \Gamma_1, \Gamma_2$  is a subgroup with the property that any
1-cocycle for $\sigma^\omega$ is co-boundary, where $\sigma$ is an arbitrary free $H$-action on a II$_1$ factor, then $\Gamma_1*_H \Gamma_2$ lies
in $\VC^\omega$ as well. But by Theorem \ref{thm.amenable-1-cocycle-approx-vanish}, the latter property is satisfied by any amenable group $H$.
\end{proof}

\begin{remark}
Note that Lemmas \ref{lem.one} and \ref{lem.two} already show that if $\Gamma$ is a countable amenable group, then no free action of $\Gamma$ on a II$_1$ factor is strongly ergodic. This was shown for $N$ having property Gamma in \cite{B90}, where one actually proves that in this case, one has $(N \rtimes \Gamma)'\cap N^\om \neq \C1$, by using results in \cite{O85}.
One can also give an alternative, more direct (``by hand'') proof of the lack of strong ergodicity by constructing projections $p\in N$ of trace $1/2$ that are almost $\Gamma$-invariant by using the F{\o}lner condition and ``local Rokhlin towers'' for $\Gamma \curvearrowright N$ and a maximality argument.
\end{remark}

Recall that an action $\al$ of a countable group $G$ on a countable pmp equivalence relation $\cS$ on $(X,\mu)$ is called \emph{outer} if $(x,\al_g(x)) \not\in \cS$ for all $g \in \Gamma \setminus \{e\}$ and a.e.\ $x \in X$.

\begin{theorem}\label{thm.approx-vanish-Cartan}
Let $B$ be a II$_1$ factor with separable predual and Cartan subalgebra $A \subset B$. Let $G$ be a countable amenable group and $\al$ an action of $G$ by automorphisms of $A \subset B$ so that the action on the associated equivalence relation $\cS$ is outer.

Every $1$-cocycle $c : G \recht \cN_B(A)$ for $\al$ is approximately a co-boundary: there exists a sequence of $a_n \in \cN_B(A)$ such that
$$\lim_n \|c_g - a_n \al_g(a_n^*) \|_2 = 0 \quad\text{for all}\;\; g \in G \; .$$
\end{theorem}
\begin{proof}
Fix a finite subset $F \subset G$ and $\eps > 0$. We will construct an element $a \in \cN_B(A)$ such that $\|c_g - a \al_g(a^*)\|_2^2 \leq 8 \eps$ for every $g \in F$.

Denote $\beta_g = (\Ad c_g) \circ \al_g$. Since the action $\al$ of $G$ on $\cS$ is outer, both $\al_g$ and $\be_g$ act freely on $A$. By the version of the Ornstein-Weiss Rokhlin theorem proved in \cite[Theorem 4.46]{KL16}, we can take $n \in \N$ and, for all $i \in \{1,\ldots,n\}$, a finite subset $T_i \subset G$ and projections $p_i,q_i \in A$ such that the following conditions hold.
\begin{itemlist}
\item The projections $\{\al_h(p_i) \mid i = 1,\ldots,n , h \in T_i \}$ are orthogonal and their sum $p$ has trace at least $1-\eps$.
\item The projections $\{\be_h(q_i) \mid i = 1,\ldots,n , h \in T_i \}$ are orthogonal and their sum $q$ has trace at least $1-\eps$.
\item For every $i \in \{1,\ldots,n\}$, we have that $\tau(p_i) = \tau(q_i)$.
\item For every $g \in F$ and $i \in \{1,\ldots,n\}$, we have that $|g T_i \cap T_i| \geq (1-\eps)|T_i|$.
\end{itemlist}
Since the equivalence relation $\cS$ is ergodic, we can choose $v_i \in B$ such that $v_i^* v_i = p_i$, $v_i v_i^* = q_i$ and $v_i A v_i^* = A q_i$, for all $i \in \{1,\ldots,n\}$. We can also choose $w \in B$ such that $w^* w = 1-p$, $ww^* = 1-q$ and $w A w^* = A(1-q)$. Define
$$a = w + \sum_{i=1}^n \sum_{h \in T_i} c_h \, \al_h(v_i) \; .$$
Note that $a$ is a sum of partial isometries normalizing $A$ and having orthogonal initial, resp.\ final projections. It then follows that $a \in \cN_B(A)$.

Fix $g \in F$ and define the projection $f \in A$ given by
$$f = \sum_{i=1}^n \sum_{h \in gT_i \cap T_i} \al_h(p_i) \; .$$
One checks that $c_g \al_g(a) f = a f$. It then follows that
$$\|c_g \al_g(a) - a\|_2^2 = \|(c_g \al_g(a) - a)(1-f)\|_2^2 \leq 4 \tau(1-f) \; .$$
We also have that
$$\tau(f) = \sum_{i=1}^n |g T_i \cap T_i| \, \tau(p_i) \geq (1-\eps) \sum_{i=1}^n |T_i| \, \tau(p_i) = (1-\eps) \, \tau(p) \geq 1-2\eps \; .$$
So, we have proven that $\|c_g - a \al_g(a^*)\|_2^2 \leq 8 \eps$ for all $g \in F$.
\end{proof}

\section{Regular inclusions containing a Cartan subalgebra}

Let $(M,\tau)$ be a tracial von Neumann algebra with separable predual and Cartan subalgebra $A \subset M$. By \cite{FM75}, we can write $A = L^\infty(X)$ and $M = L(\cR,u)$ where $\cR$ is a countable pmp equivalence relation on the standard probability space $(X,\mu)$ and $u$ is a $2$-cocycle on $\cR$ with values in the circle $\T$. There is a bijective correspondence between intermediate von Neumann algebras $A \subset B \subset M$ and subequivalence relations $\cS \subset \cR$ given by $B = L(\cS,u)$. Note that automatically, $B'\cap M \subset A' \cap M = A \subset B$, so that $B' \cap M = \cZ(B)$.

Recall that an action $\al$ of a countable group $\Gamma$ on a countable pmp equivalence relation $\cS$ on $(X,\mu)$ is called \emph{outer} if for all $g \in \Gamma \setminus \{e\}$, we have that $(x,\al_g(x)) \not\in \cS$ for a.e.\ $x \in X$.

We start by sketching the proof of the following folklore result.

\begin{proposition}\label{prop.crossed-product-equiv-rel}
In the setting introduced above, we have that the inclusion $B \subset M$ is regular if and only if the subequivalence relation $\cS \subset \cR$ is strongly normal in the sense of \cite[Definition 2.14]{FSZ88}.

In that case, writing $A \subset B$ as the direct integral of the Cartan inclusions $(A_y \subset B_y)_{y \in Y}$ with the $B_y$ being factors, we can write $M$ as the cocycle crossed product of a pmp discrete measured groupoid $\cG$ and a free cocycle action $(\al,u)$ on this field $(A_y \subset B_y)_{y \in Y}$. This means that $\cG^{(0)} = Y$ and that we are given
\begin{enumlist}
\item a measurable field of $*$-isomorphisms $\cG \ni g \mapsto \al_g : B_{s(g)} \recht B_{t(g)}$ satisfying $\al_g(A_{s(g)}) = A_{t(g)}$,
\item a measurable field of unitaries $\cG^{(2)} \ni (g,h) \mapsto u(g,h) \in \cN_{B_{t(g)}}(A_{t(g)})$,
\end{enumlist}
satisfying the conditions in Definition \ref{def.cocycle-action} and such that for every $y \in Y$, the action of the isotropy group $\Gamma_y = \{g \in \cG \mid s(g) = y = t(g)\}$ on the equivalence relation $\cS_y$ associated with $A_y \subset B_y$ is outer.
\end{proposition}
\begin{proof}
By definition, strong normality of $\cS \subset \cR$ means that $\cR$ is generated by the graphs of the elements of $\cF = \{\vphi \in [\cR] \mid \vphi \in \Aut(\cS)\}$. Whenever $\vphi \in \cF$, the corresponding unitary element $u_\vphi \in M = L(\cR,u)$ normalizes $B = L(\cS,u)$. Therefore, strong normality of $\cS \subset \cR$ implies the regularity of $B \subset M$.

Conversely, assume that $B \subset M$ is regular. Since $B' \cap M \subset A' \cap M = A \subset B$, we get that $B' \cap M = \cZ(B)$. So, as in Section \ref{sec.groupoids}, we can write $M$ as a cocycle crossed product of a pmp discrete measured groupoid $\cG$ by a free cocycle action on $B$. To conclude the proof of the proposition, we have to prove that this cocycle action can be chosen in such a way that it globally preserves $A \subset B$. For this, it suffices to prove that for every $u \in \cN_M(B)$, there exists a unitary $b \in \cU(B)$ such that $b u \in \cN_M(A)$. Fix $u \in \cN_M(B)$ and denote by $\be \in \Aut(B)$, the automorphism given by $\be(x) = u x u^*$ for all $x \in B$. Fix any $v \in \cN_M(A)$ and denote by $\al \in \Aut(A)$, the automorphism given by $\al(a) = v a v^*$ for all $a \in A$. Define $b = E_B(v u^*)$. It follows that $b \beta(a) = \al(a) b$ for all $a \in A$. Denoting by $z_v \in \cZ(B)$ the central support of the element $b^*b \in B$, we conclude that the Cartan subalgebras $\beta(A) z_v$ and $A z_v$ of $B z_v$ are unitarily conjugate. Since $A \subset M$ is regular, varying $v \in \cN_M(A)$, the join of all these central projections $z_v \in \cZ(B)$ is $1$. Therefore $\beta(A)$ and $A$ are unitarily conjugate. So we can choose $b \in \cU(B)$ such that $b \beta(A) b^* = A$ and thus $bu \in \cN_M(A)$.
\end{proof}

Ignoring the scalar cocycles in Proposition \ref{prop.crossed-product-equiv-rel}, we can also write the equivalence relation $\cR$ as the semidirect product of the subequivalence relation $\cS$ and an outer cocycle action of the groupoid $\cG$ by automorphisms of $\cS$.

The main goal of this section is to prove the following analogues of Theorems \ref{thm.two-cocycle-vanishing} and \ref{thm.classification-amenable}. These results were so far only known for amenable groups (see \cite[Theorem 3.4]{FSZ88}, \cite[Theorem 3.4]{BG84} and \cite{GS86}).

\begin{theorem}\label{thm.two-cocycle-vanishing-Cartan}
Let $\cG$ be a discrete measured groupoid with $Y = \cG^{(0)}$ and $(A_y \subset B_y)_{y \in Y}$ a measurable field of Cartan subalgebras in II$_1$ factors with separable predual. Assume that $\cG$ is amenable.

When $(\al,u)$ is a free cocycle action of $\cG$ on $(A_y \subset B_y)_{y \in Y}$, the cocycle $u$ is a co-boundary: there exists a measurable field of unitaries $\cG \ni g \mapsto w_g \in \cN_{B_{t(g)}}(A_{t(g)})$ such that
$$u(g,h) = \al_g(w_h^*) \, w_g^* \, w_{gh} \quad\text{for all $(g,h) \in \cG^{(2)}$.}$$
\end{theorem}

Applying Theorem \ref{thm.two-cocycle-vanishing-Cartan} in a setup without scalar cocycles, we obtain the following result: if $\cR$ is a countable pmp equivalence relation with strongly normal subequivalence relation $\cS \subset \cR$ and if the quotient groupoid $\cG$ is amenable, we can write $\cR$ as the semidirect product of $\cS$ by a free action of $\cG$ on $\cS$.

\begin{proof}[Proof of Theorem \ref{thm.two-cocycle-vanishing-Cartan}]
The proof follows exactly the same lines as the proof of Theorem \ref{thm.two-cocycle-vanishing}. We need the following two ingredients.
\begin{itemlist}
\item Vanishing of $2$-cocycles for actions of amenable groups: if $B$ is a separable II$_1$ factor with Cartan subalgebra $A \subset B$ and associated ergodic type II$_1$ equivalence relation $\cS$ and if $(\al,u)$ is a cocycle action of an amenable group $\Gamma$ on $A \subset B$ such that the corresponding action on $\cS$ is outer, then the $2$-cocycle $u$ is a co-boundary. Theorem 3.4 in \cite{FSZ88} provides $2$-cocycle vanishing for the corresponding outer cocycle action on the equivalence relation $\cS$. So, we may perturb $(\al,u)$ such that $u(g,h) \in \cU(A)$ for all $g,h \in \Gamma$. Writing $A = L^\infty(X,\mu)$, we now have that $\Gamma \actson (X,\mu)$ is a free action and we can view $u$ as a scalar $2$-cocycle on the associated orbit equivalence relation. Since this equivalence relation is hyperfinite, the scalar $2$-cocycle $u$ is a co-boundary.

\item Approximate vanishing of $1$-cocycles for actions of amenable groups: this is Theorem \ref{thm.approx-vanish-Cartan} above.
\end{itemlist}
We then follow the same method as in the proof of Theorem \ref{thm.two-cocycle-vanishing}. Applying Theorem \ref{thm.equivariant-section}, the result follows.
\end{proof}

Below, we then prove the following classification result.

\begin{theorem}\label{thm.classification-amenable-Cartan}
Let $R$ be the hyperfinite II$_1$ factor with Cartan subalgebra $A \subset R$. The intermediate subalgebras $A \subset B \subset R$ such that $B \subset R$ is regular are completely classified by the type of $B$ and the associated discrete measured groupoid.

More precisely, given for $i=1,2$, von Neumann subalgebras $A_i \subset B_i \subset R$ with $A_i \subset R$ Cartan and $B_i \subset R$ regular, there exists an automorphism $\theta \in \Aut(R)$ satisfying $\theta(B_1) = B_2$ and $\theta(A_1) = A_2$ if and only if $B_1$ and $B_2$ have the same type and $\cG_1 \cong \cG_2$.
\end{theorem}

As in Remark \ref{rem.what-if-non-factorial}, it is possible to formulate a version of Theorem \ref{thm.classification-amenable-Cartan} without assuming factoriality, i.e.\ for von Neumann subalgebras of an arbitrary amenable tracial von Neumann algebra.

In order to deduce Theorem \ref{thm.classification-amenable-Cartan} from the vanishing of $2$-cohomology in Theorem \ref{thm.two-cocycle-vanishing-Cartan}, we need to show that an amenable discrete measured groupoid $\cG$ with $\cG^{(0)} = Y$ has only one outer measure preserving action (up to cocycle conjugacy) on the field $(\cS_y)_{y \in Y}$ of ergodic hyperfinite type II$_1$ equivalence relations.

For outer actions of amenable groups on the ergodic hyperfinite type II$_1$ equivalence relation, this was proved in \cite[Theorem 3.4]{BG84} and this should be considered as an equivalence relation version of Ocneanu's theorem \cite{O85}. For completeness, we include a short proof of this result (see Theorem \ref{thm.ocneanu-eq-rel} and Lemma \ref{lem.golodets} below). We then establish the groupoid case in Theorem \ref{thm.ocneanu-eq-rel-groupoid} (by invoking once more Theorem \ref{thm.equivariant-section}). Once Theorem \ref{thm.ocneanu-eq-rel-groupoid} proven, Theorem \ref{thm.classification-amenable-Cartan} follows from Theorem \ref{thm.two-cocycle-vanishing-Cartan}.

\begin{theorem}[{See \cite[Theorem 3.4]{BG84}}]\label{thm.ocneanu-eq-rel}
Let $B$ be the hyperfinite II$_1$ factor with Cartan subalgebra $A \subset B$ and associated equivalence relation $\cS$. Let $G$ be a countable amenable group and $\al,\be$ actions of $G$ on $A \subset B$ such that the induced action on $\cS$ is outer.

Then there exists a $*$-automorphism $\theta \in \Aut(B)$ and a $1$-cocycle $g \mapsto c_g \in \cN_B(A)$ for the action $\al$ such that $\theta(A) = A$ and
$$\theta \circ \be_g \circ \theta^{-1} = (\Ad c_g) \circ \al_g \quad\text{for all}\;\; g \in G \; .$$
\end{theorem}

We prove Theorem \ref{thm.ocneanu-eq-rel} below, as an immediate consequence of the following lemma, which is a special case of \cite[Theorem 1.5]{GS86}.

\begin{lemma}[{See \cite[Theorem 1.5]{GS86}}]\label{lem.golodets}
Let $\cR$ be a countable pmp ergodic hyperfinite equivalence relation on the standard probability space $(X,\mu)$. Let $\al, \be : \cR \recht G$ be $1$-cocycles with values in a countable group $G$.
Assume that $\al,\be$ are surjective and that their kernels are ergodic subequivalence relations of $\cR$. Then, there exists $\theta \in \Aut(\cR)$ such that
$$\al(x,y) = \be(\theta(x),\theta(y)) \quad\text{for a.e. $(x,y) \in \cR$.}$$
\end{lemma}
\begin{proof}
In the course of the proof, we again tacitly assume that all statements are valid up to sets of measure zero. We call a $\si$-algebra $\cB$ of Borel subsets of $X$ \emph{atomic} if $\cB$ is generated by countably many minimal elements, called atoms. Note that $\cB$ is atomic if and only if the algebra of bounded $\cB$-measurable functions is an atomic von Neumann subalgebra of $L^\infty(X)$.

We call \emph{decomposition} of $\cR$ any $n$-tuple $P=(\vphi_1,\ldots,\vphi_n)$ of elements of $[[\cR]]$ such that the range sets $X_i = R(\vphi_i)$ form a partition of $X$ with $D(\vphi_i) = X_1$ for all $i$ and $\vphi_1(x) = x$ for all $x \in X_1$. Note that the graphs of the elements $\vphi_i$ generate a finite subequivalence relation of $\cR$.

We say that such a decomposition is \emph{compatible} with an atomic $\si$-algebra $\cB_0$ and a cocycle $\al : \cR \recht G$ if the following conditions hold.
\begin{itemlist}
\item For every $i = 1,\ldots,n$, the set $X_i$ belongs to $\cB_0$ and a subset $\cU$ of $X_1$ belongs to $\cB_0$ if and only if $\vphi_i(\cU)$ belongs to $\cB_0$.
\item For every $g \in G$ and $i = 1,\ldots,n$, the set $\{x \in X_1 \mid \al(x,\vphi_i(x)) = g \}$ belongs to $\cB_0$.
\end{itemlist}
We denote by $\cG(P,\cB_0)$ the pseudogroup of partial transformations $\vphi$ with $D(\vphi), R(\vphi) \in \cB_0$ and with the graph of $\vphi$ contained in the finite subequivalence relation generated by $P$.

Assume now that we have two compatible decompositions: $P=(\vphi_1,\ldots,\vphi_n)$, $\cB_0$, $\al$, and $Q=(\psi_1,\ldots,\psi_n)$, $\cD_0$, $\be$. Write $X_1 = D(\vphi_i)$ and $Y_1 = D(\psi_i)$. We call \emph{isomorphism} between these two compatible decompositions any measure preserving isomorphism $\theta_0 : \cB_0 \recht \cD_0$ satisfying the following properties.
\begin{itemlist}
\item $\theta_0(X_1) = Y_1$ and $\theta_0(R(\vphi_i)) = R(\psi_i)$ for all $i$.
\item $\theta_0(\vphi_i(\cU)) = \psi_i(\theta(\cU))$ whenever $\cU \in \cB_0$ and $\cU \subset X_1$.
\item For every $g \in G$ and every $i$, $\theta_0$ maps the set $\{x \in X_1 \mid \al(x,\vphi_i(x)) = g \}$ to the set $\{y \in Y_1 \mid \be(y,\psi_i(y)) = g\}$.
\end{itemlist}

Further assume that $(\vphi'_1,\ldots,\vphi'_m)$ is decomposition of $\cR|_{X_1}$. We then define the refined decomposition $P'$ of $\cR$ consisting of the maps $\vphi_i \vphi'_j$. Let $\cC$ be any atomic $\si$-algebra containing $\cB_0$. We make the following two claims.
\begin{enumlist}
\item There exists an atomic $\si$-algebra $\cB_1$ containing $\cC$ such that the decomposition $P'$ is compatible with $\cB_1$ and $\al$.
\item There exists an atomic $\si$-algebra $\cD_1$ containing $\cD_0$ and $\cR|_{Y_1}$ admits a decomposition $(\psi'_1,\ldots, \psi'_m)$ such that the refined decomposition $Q'$ consisting of the maps $\psi_i \psi'_j$ is compatible with $\cD_1$ and $\be$, and such that $\theta_0$ can be extended to an isomorphism between $P',\cB_1,\al$ and $Q',\cD_1,\be$.
\end{enumlist}

Once these claims are proven, we inductively construct as follows finer and finer compatible decompositions $P_n,\cB_n,\al$ and $Q_n,\cD_n,\be$ and isomorphisms $\theta_n$ between them. At step 0, we take the trivial decomposition consisting of the identity map on $X$ and the trivial $\si$-algebras consisting of $\emptyset$ and $X$. At the $n$-th step for $n$ odd, using the first part of the claim, we choose the refinements $P_n$ and $\cB_n$ in such a way that the pseudogroup $\cG(P_n,\cB_n)$ approximately contains a given finite subset of $[[\cR]]$. This is possible because $\cR$ is hyperfinite. We then use the second part of the claim to pick $Q_n,\cD_n$ and to extend the isomorphism $\theta_{n-1}$ to $\theta_n$. At the next even step, we first choose the refinements $Q_{n+1},\cD_{n+1}$ in such a way that $\cG(Q_{n+1},\cD_{n+1})$ approximately contains a given finite subset of $[[\cR]]$. We then pick $P_{n+1},\cB_{n+1}$ and extend the isomorphism $\theta_n$ to $\theta_{n+1}$.

Taking the inductive limit, we find a pmp isomorphism $\theta : X \recht X$ with the following properties.
$$\theta \circ \vphi_i = \psi_i \circ \theta \quad\text{and}\quad \theta\bigl(\{x \in D(\vphi_i) \mid \al(x,\vphi_i(x)) = g \}\bigr) = \{y \in D(\psi_i) \mid \be(y,\psi_i(y)) = g \}$$
whenever $\vphi_i,\psi_i$ belong to one of the decompositions $P_n,Q_n$. The graphs of all $\vphi_i$, resp.\ all $\psi_i$, generate the entire equivalence relation $\cR$. The first property thus implies that $\theta \in \Aut(\cR)$. The second property says that $\al(x,\vphi_i(x)) = \be(\theta(x),\psi_i(\theta(x)))$. In combination with the first property, this means that $\al(x,y) = \be(\theta(x),\theta(y))$ for all $(x,y) \in \cR$ and the theorem is proven.

It remains to prove the two claims. To prove the first claim, it suffices to show the following: if $P=(\eta_1,\ldots,\eta_r)$ is a decomposition of $\cR$ and $\cC$ is an atomic $\si$-algebra, there exists an atomic $\si$-algebra $\cB$ containing $\cC$ such that $P$, $\cB$ and $\al$ are compatible. Write $Z_i = R(\eta_i)$. First refine $\cC$ so that $Z_i \in \cC$ for all $i$. For every $i$, define as follows the atomic $\si$-algebras $\cC_i$ and $\cD_i$ on $Z_1$. Define $\cC_i = \eta_i^{-1}(\cC|_{Z_i})$ and define $\cD_i$ generated by the partition of $Z_1$ into the subsets
$$Z_1 = \bigsqcup_{g \in G} \{x \in Z_1 \mid \al(x,\eta_i(x)) = g\} \; .$$
The finitely many atomic $\si$-algebras $\cC_i$ and $\cD_i$ generate an atomic $\si$-algebra on $Z_1$ that we denote as $\cE$. Then define $\cB$ as the $\si$-algebra generated by the $\si$-algebras $\eta_i(\cE)$ on $Z_i$. This concludes the proof of the first claim.

Since $\theta_0(X_1) = Y_1$ and thus $\theta_0({\cB_0}|_{X_1}) = {\cD_0}|_{Y_1}$, we can choose an atomic $\si$-algebra $\cE$ on $Y_1$ such that the restriction of $\theta_0$ to $\cB_0|_{X_1}$ extends to a measure preserving $\si$-algebra isomorphism $\theta : {\cB_1}|_{X_1} \recht \cE$. Write $X'_j = R(\vphi'_j)$ and write $X'_1 = \bigsqcup_n \cU_n$ where the $\cU_n$ are atoms of the $\si$-algebra $\cB_1$. Since the $\cU_n$ are atoms, we find elements $g_{j,n} \in G$ such that $\al(x,\vphi'_j(x)) = g_{j,n}$ for all $j$ and all $x \in \cU_n$. Since the cocycle $\be$ is surjective, we find non-negligible $\psi_{j,n} \in [[\cR]]$ such that $\be(y,\psi_{j,n}(y)) = g_{j,n}$. Since the kernel of $\be$ is ergodic, we can make this choice of $\psi_{j,n}$ such that $D(\psi_{j,n}) = \theta(\cU_n)$ and $R(\psi_{j,n}) = \theta(\vphi'_j(\cU_n))$. So, for a fixed $j \in \{1,\ldots,m\}$, the domains $D(\psi_{j,n})$, $n \in \N$, form a partition of $\theta(X'_1)$ and the ranges $R(\psi_{j,n})$, $n \in \N$, form a partition of $\theta(X'_j)$. We define $\psi'_j$ by gluing together the $\psi_{j,n}$, $n \in \N$. This concludes the proof of the second claim.
\end{proof}

\begin{proof}[Proof of Theorem \ref{thm.ocneanu-eq-rel}]
Write $M = B \rtimes_\al G$ and $N = B \rtimes_\beta G$. Denote by $\cR$ the hyperfinite ergodic type II$_1$ equivalence relation on the standard probability space $(X,\mu)$. Choose $*$-isomorphisms $\pi_1 : M \recht L(\cR)$ and $\pi_2 : N \recht L(\cR)$ satisfying $\pi_1(A) = L^\infty(X) = \pi_2(A)$. Then, $\pi_1(B) = L(\cS_1)$ and $\pi_2(B) = L(\cS_2)$ where the $\cS_i \subset \cR$ are subequivalence relations that arise as the kernel of the natural cocycles $\cR \recht G$. By Lemma \ref{lem.golodets}, there exists an automorphism $\gamma \in \Aut(\cR)$ satisfying $\gamma(\cS_2) = \cS_1$. We still denote by $\gamma$ the induced automorphism of $L(\cR)$. Then, $\theta = \pi_1^{-1} \circ \gamma \circ \pi_2$ is a $*$-isomorphism $\theta : N \recht M$ satisfying $\theta(A) = A$ and $\theta(B) = B$. Defining $c_g \in \cN_B(A)$ such that $\theta(u_g) = c_g u_g$ and restricting $\theta$ to $B$, we have found the required cocycle conjugacy between $\al$ and $\be$.
\end{proof}

In order to prove a groupoid version of Theorem \ref{thm.ocneanu-eq-rel}, we first need the following lemma on approximate innerness of a cocycle self-conjugacy.

\begin{lemma}\label{lem.approx-inner-cartan}
Let $B$ be the hyperfinite II$_1$ factor with Cartan subalgebra $A \subset B$ and associated equivalence relation $\cS$. Let $G$ be a countable amenable group and $\al$ an action of $G$ on $A \subset B$ such that the induced action on $\cS$ is outer.

Whenever $\theta \in \Aut(A \subset B)$ and $c_g \in \cN_B(A)$ is a $1$-cocycle for the action $\al$ satisfying $\theta \al_g \theta^{-1} = (\Ad c_g) \al_g$ for all $g \in G$, there exists a sequence $w_n \in \cN_B(A)$ such that
$$\Ad w_n \recht \theta \;\;\text{in $\Aut(B)$ and}\;\; \lim_n \| c_g - w_n \al_g(w_n^*) \|_2 = 0 \;\;\text{for all $g \in G$.}$$
\end{lemma}
\begin{proof}
Write $M = B \rtimes_\al G$ and denote by $(u_g)_{g \in G}$ the canonical unitary operators in this crossed product. We call \emph{compatible von Neumann subalgebra} of $M$ any von Neumann subalgebra $M_0 \subset M$ that is generated by an atomic von Neumann subalgebra $A_0 \subset A$ and finitely many partial isometries $v_1,\ldots,v_n \in M$ satisfying the following properties.
\begin{itemlist}
\item The projections $p_i = v_i v_i^*$ belong to $A_0$ and sum up to $1$. We have $v_i^* v_i = p_1$ for every $i$.
\item We have $v_i A_0 v_i^* = A_0 p_i$ and $v_i A v_i^* = A p_i$.
\item There exist projections $p_{i,g} \in A_0$ such that for every $i$, we have that
$$\sum_{g \in G} p_{i,g} = p_1 \quad \text{and} \quad v_i p_{i,g} \in B u_g \;\;\text{for all $g \in G$.}$$
\end{itemlist}

We claim that $M$ can be generated by an increasing sequence of compatible von Neumann subalgebras. To prove this claim, let $\cR$ be the hyperfinite ergodic type II$_1$ equivalence relation on the standard probability space $(X,\mu)$. Choose a $*$-isomorphism $\pi : M \recht L(\cR)$ satisfying $\pi(A) = L^\infty(X)$. Then, $\pi(B) = L(\cS)$ for an ergodic subequivalence relation $\cS \subset \cR$ that arises as the kernel of a surjective cocycle $\cR \recht G$. Under the isomorphism $\pi$, the compatible von Neumann subalgebras of $M$ correspond to the compatible decompositions of $\cR$ introduced in the proof of Lemma \ref{lem.golodets}. Our claim then follows from claim 1 in that proof.

Extend $\theta$ to an automorphism of $M$ by defining $\theta(u_g) = c_g u_g$. To prove the lemma, we need to prove that there exist elements $w_n \in \cN_B(A)$ such that $\Ad w_n \recht \theta$ in $\Aut(M)$. Let $M_0 \subset M$ be a compatible von Neumann subalgebra generated by $A_0 \subset A$ and $v_1,\ldots,v_n$ as above. Given the claim above, it suffices to show that there exists an element $w \in \cN_B(A)$ such that $\theta(a) = w a w^*$ for all $a \in A_0$ and $\theta(v_i) = w v_i w^*$ for all $i \in \{1,\ldots,n\}$.

First note that $\theta(A_0)$ is a discrete von Neumann subalgebra of $A$ whose minimal projections have the same trace values as those of $A_0$. Since $B$ is a factor, we can pick $w_1 \in \cN_B(A)$ such that $w_1^* \theta(a) w_1 = a$ for all $a \in A_0$. Write $\theta_1 = (\Ad w_1^*) \circ \theta$. We claim that
$$w_2 = \sum_{j=1}^n \theta_1(v_j) v_j^*$$
belongs to $\cN_B(A)$ and satisfies $w_2 a = a w_2$ for all $a \in A_0$ and $w_2 v_i = \theta_1(v_i) w_2$ for all $i \in \{1,\ldots,n\}$. Once this claim is proven, the element $w = w_1 w_2$ in $\cN_B(A)$ satisfies the required properties.

Define the projections $q_{i,g} = v_i p_{i,g} v_i^*$. By our assumptions, $q_{i,g} \in A_0$ and $q_{i,g} v_i = v_i p_{i,g}$, while $q_{i,g} v_j = 0$ if $i \neq j$. Also,
$$\sum_{i=1}^n \sum_{g \in G} q_{i,g} = 1 \; .$$
Since $\theta_1$ is the identity on $A_0$, it follows that $w_2 q_{i,g} = \theta_1(v_i p_{i,g}) (v_i p_{i,g})^*$. Since $v_i p_{i,g} \in B u_g$ and thus also $\theta_1(v_i p_{i,g}) \in Bu_g$, it follows that $w_2 \in B$.

Since $\theta_1(A) = A$ and since the $v_j$ normalize $A$, we have for every $j \in \{1,\ldots,n\}$ that
$$w_2 \, A p_j \, w_2^* = \theta_1(v_j) v_j^* A v_j \theta_1(v_j^*) = \theta_1(v_j) \, A p_1 \, \theta_1(v_j^*) = \theta_1 (v_j A v_j^*) = \theta_1(A p_j) = A p_j \; .$$
So, $w_2 \in \cN_B(A)$. The formula $w_2 a = a w_2$ for all $a \in A_0$ follows because $\theta_1$ is the identity on $A_0$ and the $v_j$ normalize $A_0$. The formula $w_2 v_i = \theta_1(v_i) w_2$ holds by definition.
\end{proof}

\begin{theorem}\label{thm.ocneanu-eq-rel-groupoid}
Let $\cG$ be an amenable discrete measured groupoid with $\cG^{(0)} = Y$. Let $(A_y \subset B_y)_{y \in Y}$ be a measurable field of Cartan subalgebras with each $B_y$ being isomorphic to the hyperfinite II$_1$ factor. Let $\al,\be$ be free actions of $\cG$ on $(A_y \subset B_y)_{y \in Y}$. Then, $\al$ and $\be$ are cocycle conjugate: there exists a measurable field of $*$-automorphisms $\theta_y \in \Aut(B_y)$ satisfying $\theta_y(A_y) = A_y$ and a measurable field $\cG \ni g \mapsto c_g \in \cN_{B_{t(g)}}(A_{t(g)})$ such that
$$\theta_{t(g)} \circ \be_g \circ \theta_{s(g)}^{-1} = (\Ad c_g) \circ \al_g \quad\text{and}\quad c_{hk} = c_h \, \al_h(c_k)$$
for all $g \in \cG$ and $(h,k) \in \cG^{(2)}$.
\end{theorem}
\begin{proof}
As in the proof of Theorem \ref{thm.classification-amenable}, we denote by $\Gamma_y = \{g \in \cG \mid s(g) = y = t(g)\}$ the isotropy groups of $\cG$ and we view $\cG$ as the semidirect product of the field $(\Gamma_y)_{y \in Y}$ of groups and the measurable family of group isomorphisms $\delta_{(y,z)} : \Gamma_z \recht \Gamma_y$ where $(y,z) \in \cR$ and $\cR$ is an amenable nonsingular countable equivalence relation on $Y$.

Denote by $P_y$ the Polish space of cocycle conjugacies between the actions $(\al_g)_{g \in \Gamma_y}$ and $(\be_g)_{g \in \Gamma_y}$ on $A_y \subset B_y$. More precisely,
\begin{align*}
P_y := \bigl\{\bigl(\theta,(c_g)_{g \in \Gamma_y}\bigr) \bigm| \; & \theta \in \Aut(B_y) \;\; ,  \;\; \theta(A_y) = A_y \;\; ,  \;\; c_g \in \cN_{B_y}(A_y) , \\ & c_{gh} = c_g \al_g(c_h)  \;\; ,  \;\; \theta \be_g \theta^{-1} = (\Ad c_g) \al_g \;\;\text{for all}\;\; g,h \in \Gamma_y \bigr\} \; .
\end{align*}
The topology on $P_y$ is given by the usual topology on $\Aut(B_y)$ and the topology of pointwise $\|\cdot\|_2$-convergence on the space of $1$-cocycles. By Theorem \ref{thm.ocneanu-eq-rel}, $P_y$ is nonempty. Writing $G_y = \cN_{B_y}(A_y)$, we have a natural action $G_y \actson P_y$. By Lemma \ref{lem.approx-inner-cartan}, this action has dense orbits.

Define the measurable field of continuous group isomorphisms $\cR \ni (y,z) \mapsto \al_{(y,z)} : G_z \recht G_y$. Also define the homeomorphisms
$\cR \ni (y,z) \mapsto \gamma_{(y,z)} : P_z \recht P_y$ given by mapping $(\theta,(c_g)_{g \in \Gamma_z}) \in P_z$ to the automorphism $\al_{(z,y)} \theta \beta_{(y,z)}$ of $B_z$ and the $1$-cocycle $g \mapsto \al_{(z,y)}(c_{\delta_{(y,z)}(g)})$. By Theorem \ref{thm.equivariant-section}, we find a measurable field of cocycle conjugacies $\pi_y = (\theta_y,(c_g)_{g \in \Gamma_y})$ and a $1$-cocycle $\cR \ni (y,z) \mapsto c_{(y,z)} \in G_y$ such that $\pi_y = c_{(y,z)} \cdot \pi_z$.
This precisely means that the $(c_g)_{g \in \Gamma_y}$ and $(c_{(y,z)})_{(y,z) \in \cR}$ combine into a $1$-cocycle for $\al$, which together with $(\theta_y)_{y \in Y}$ forms the required cocycle conjugacy.
\end{proof}

\section{Remarks on treeability and cost of equivalence relations}\label{sec.tree-cost}

The class $\VC$, introduced in \cite{Popa18}, is defined as the class of countable groups $\Gamma$ satisfying $2$-cohomology vanishing for cocycle actions on II$_1$ factors. We still denote by $\VC$ the class of discrete measured groupoids $\cG$ satisfying the conclusion of Theorem \ref{thm.two-cocycle-vanishing}. Similarly, we denote by $\VCCartan$ the class of discrete measured groupoids satisfying the conclusion of Theorem \ref{thm.two-cocycle-vanishing-Cartan}.

As in \cite[Remarks 4.5]{Popa18}, one may speculate that for a pmp discrete measured groupoid $\cG$, we have that $\cG \in \VC$ iff $\cG \in \VCCartan$ iff $\cG$ is treeable, in the sense discussed in this section. We elaborate on these speculations below and prove that a groupoid in $\VCCartan$ must be treeable (see Proposition \ref{prop.VCCartan-implies-treeable}). This also leads us to a new notion of treeability for countable pmp equivalence relations, which we call \emph{weak treeability}. In Proposition \ref{prop.all-kinds-treeability}, we relate the well known open problem whether every treeable countable group $\Gamma$ is strongly treeable to the question whether every weakly treeable equivalence relation is treeable. Analogously, we introduce a \emph{fixed price question} for equivalence relations and connect it to the same question for countable groups in Remark \ref{rem.fixed-price}.

We first formulate a definition and then explain the required terminology.

\begin{definition}\label{def.treeable-groupoid}
A pmp discrete measured groupoid $\cG$ is said to be treeable if $\cG$ admits a free pmp action on a field of standard probability spaces such that the orbit equivalence relation is treeable.
\end{definition}

More concretely, $\cG$ is treeable if and only if there exists a countable pmp equivalence relation $\cS$ on $(Y,\eta)$ that is treeable, a measure preserving factor map $\pi_0 : Y \recht \cG^{(0)}$ and a nonsingular factor map $\pi : \cS \recht \cG$ satisfying the following properties
\begin{align*}
& s(\pi(x,y)) = \pi_0(y) \;\; , \;\; t(\pi(x,y)) = \pi_0(x) \;\; , \;\; \pi(x,x) = \pi_0(x) \quad\text{for a.e.\ $(x,y) \in \cS$,}\\
& \pi(x,y) \, \pi(y,z) = \pi(x,z) \quad\text{for a.e.\ $(x,y), (y,z) \in \cS$,}\\
& \text{$\pi$ maps $\cS \cdot y \times \{y\}$ bijectively onto $s^{-1}(\pi_0(y))$ for a.e.\ $y \in Y$.}
\end{align*}

Recall that a countable equivalence relation $\cS$ on $(Y,\eta)$ is called treeable if there exists a Borel graph $\cT$ on $Y$ with $\cT \subset \cS$ and such that for a.e.\ $y \in Y$, the connected component of $y$ in $\cT$ is a tree with vertex set $\cS \cdot y$.

\begin{remark}\label{rem.weakly-treeable}
Definition \ref{def.treeable-groupoid} is entirely analogous to the definition of a treeable group. Nevertheless, this definition is potentially confusing: a countable pmp equivalence relation $\cR$  can also be viewed as a groupoid. The treeability of $\cR$ as a groupoid, in the sense of Definition \ref{def.treeable-groupoid}, is in principle weaker than the treeability of $\cR$ as an equivalence relation.

We therefore say that a pmp equivalence relation $\cR$ is weakly treeable if it is treeable as a groupoid. So the countable pmp equivalence relation $\cR$ on $(X,\mu)$ is weakly treeable if and only if there exists a countable pmp equivalence relation $\cS$ on $(Y,\eta)$ that is treeable and a measure preserving factor map $\pi : Y \recht X$ mapping $\cS \cdot y$ bijectively onto $\cR \cdot \pi(y)$ for a.e.\ $y \in Y$.
\end{remark}

In the same way as the Connes-Jones cocycles of \cite{CJ84} were used in \cite[Section 4]{Popa18} to obtain restrictions on the countable groups in $\VC$, we can prove the following result.

\begin{proposition}\label{prop.VCCartan-implies-treeable}
If a pmp discrete measured groupoid $\cG$ belongs to $\VCCartan$, then $\cG$ is treeable in the sense of Definition \ref{def.treeable-groupoid}.
\end{proposition}
\begin{proof}
Denote by $(X,\mu)$ the probability space of units of $\cG$. Choose, up to measure zero, countably many subsets $\cU_n \subset \cG$ such that $s|_{\cU_n}$ and $t|_{\cU_n}$ are bijections from $\cU_n$ onto $X$ and such that $\cG = \bigcup_n \cU_n$. Denote $\Gamma = \F_\infty$ with free generators $a_n \in \Gamma$, $n \in \N$. Choose a free pmp mixing action $\Gamma \actson (Y,\eta)$, e.g.\ a Bernoulli action. For every $n$, define the measure preserving automorphism $\vphi_n \in \Aut(X,\mu)$ given by $\vphi_n = t \circ (s|_{\cU_n})^{-1}$. Define the action $\Gamma \actson (X,\mu)$ given by
$$a_{2n} \cdot x = \vphi_n(x) \quad\text{and}\quad a_{2n+1} \cdot x = x \quad\text{for all}\;\; n \in \N, x \in X \; .$$
Then consider the diagonal action $\Gamma \actson X \times Y$ and denote by $\cS$ the orbit equivalence relation. Denote by $\pi : \cS \recht \cG$ the unique groupoid morphism satisfying
$$\pi\bigl((x,y),(x,y)\bigr) = x \;\; , \;\; \pi\bigl(a_{2n} \cdot (x,y),(x,y)\bigr) = (s|_{\cU_n})^{-1}(x) \quad\text{and}\quad \pi\bigl(a_{2n+1} \cdot (x,y),(x,y)\bigr) = x \; .$$
The kernel of $\pi$ is given by a field of orbit equivalence relations $\bigl(\cR(\Lambda_x \actson Y)\bigr)_{x \in X}$, where $a_{2n+1} \in \Lambda_x$ for all $x \in X$, $n \in \N$. Then $\pi$ induces a cocycle action of $\cG$ on this field of ergodic type II$_1$ equivalence relations.

Since $\cG \in \VCCartan$, this cocycle action can be chosen to be a genuine action. This means that we find a Borel family of measure preserving automorphisms $\cG \ni g \mapsto \theta_g \in \Aut(Y,\eta)$ satisfying the following properties.
$$\theta_g \circ \theta_h = \theta_{gh} \;\; , \;\; (t(g),\theta_g(y)) \sim_\cS (s(g),y) \;\; , \;\; \pi\bigl( (t(g),\theta_g(y)) , (s(g),y) \bigr) = g$$
for all $(g,h) \in \cG^{(2)}$ and $y \in Y$.

It follows that $\cS_0 := \bigl\{ \bigl((t(g),\theta_g(y)) , (s(g),y)\bigr) \mid g \in \cG \bigr\}$ is a subequivalence relation of $\cS$. Since $\cS$ is treeable, it follows from \cite[Th\'{e}or\`{e}me 5]{G99} that also $\cS_0$ is treeable. The restriction of $\pi$ maps $\cS_0$ onto $\cG$ and maps $\cS \cdot (x,y) \times \{(x,y)\}$ bijectively onto $s^{-1}(x)$. By definition, $\cG$ is treeable.
\end{proof}

Recall that a countable group $\Gamma$ is called strongly treeable if for \emph{every} free pmp action $\Gamma \actson (X,\mu)$, the orbit equivalence relation is treeable. It is a wide open problem whether every treeable group is strongly treeable.

\begin{proposition}\label{prop.all-kinds-treeability}
The following two statements are equivalent (see Remark \ref{rem.weakly-treeable} for a disambiguation of terminology).
\begin{enumlist}
\item Every treeable countable group is strongly treeable.
\item If a countable pmp equivalence relation $\cR$ is weakly treeable and is the orbit relation of a free action of a countable group, then $\cR$ is treeable.
\end{enumlist}
\end{proposition}

If one believes that there exist treeable groups that are not strongly treeable, it could therefore be easier to first find weakly treeable, but non treeable equivalence relations.

\begin{proof}
1 $\Rightarrow$ 2. Assume that $\cR = \cR(\Gamma \actson X)$ for some free pmp action $\Gamma \actson (X,\mu)$ and assume that $\cR$ is weakly treeable. Choose a pmp action of $\cR$ such that its orbit equivalence relation $\cS$ on $(Y,\eta)$ is treeable. Denote by $\pi : Y \recht X$ the corresponding measure preserving factor map that maps $\cS$-orbits bijectively onto $\cR$-orbits. Since $\pi$ is bijective on orbits, there is a unique action $\Gamma \actson Y$ such that $g \cdot y \in \cS \cdot y$ for all $g \in \Gamma$, $y \in Y$ and $\pi(g \cdot y) = g \cdot \pi(y)$. By construction, the action $\Gamma \actson (Y,\eta)$ is free and pmp, and $\cS$ is its orbit equivalence relation.

Since $\cS$ is a treeable equivalence relation, it follows that $\Gamma$ is a treeable group. Since 1 holds, $\Gamma$ is strongly treeable and it follows that also $\cR$ is a treeable equivalence relation.

2 $\Rightarrow$ 1. Let $\Gamma$ be a countable group and let $\Gamma \actson (X,\mu)$, $\Gamma \actson (Y,\eta)$ be free pmp actions. Assume that the orbit equivalence relation $\cR(\Gamma \actson X)$ is treeable. We have to prove that also $\cR(\Gamma \actson Y)$ is a treeable equivalence relation. Using the factor map $X \times Y \recht X$, the treeing of $\cR(\Gamma \actson X)$ can be lifted to a treeing of the orbit equivalence relation $\cR(\Gamma \actson X \times Y)$ for the diagonal action. Using the factor map $X \times Y \recht Y$, we conclude that $\cR(\Gamma \actson Y)$ is weakly treeable. Since 2 holds, $\cR(\Gamma \actson Y)$ is treeable.
\end{proof}

In the proofs of Theorems \ref{thm.two-cocycle-vanishing}, \ref{thm.two-cocycle-vanishing-Cartan} and \ref{thm.ocneanu-eq-rel-groupoid}, we used the following result: if $\cG$ is a discrete measured groupoid such that the associated countable equivalence relation $\cR$ on $\cG^{(0)}$ is amenable, then $\cG$ can be written as the semidirect product of $\cR$ acting on the field of isotropy groups of $\cG$. Also this is a $2$-cohomology vanishing result, for cocycle actions on fields of countable groups. We finally prove that the same $2$-cohomology holds for arbitrary treeable equivalence relations and is a characterization of treeability.

\begin{proposition}\label{prop.semidirect-product}
Let $\cR$ be a countable nonsingular equivalence relation on the standard probability space $(X,\mu)$. The following two statements are equivalent.
\begin{enumlist}
\item $\cR$ is treeable.
\item Every discrete measured groupoid $\cG$ with $\cG^{(0)} = X$ and with $\cR = \{(t(g),s(g)) \mid g \in \cG\}$ can be written as a semidirect product of $\cR$ acting on the field of isotropy groups of $\cG$.
\end{enumlist}
\end{proposition}
\begin{proof}
1 $\Rightarrow$ 2. Since $\cR$ is treeable, we can decompose $\cR$ as a free product of amenable sub\-equivalence relations. On each of these, the quotient map $\cG \recht \cR$ can be lifted as a groupoid homomorphism and these lifts can be combined into a well defined groupoid homomorphism $\cR \recht \cG$ by freeness.

2 $\Rightarrow$ 1. Denote $\Gamma = \F_\infty$ and choose a nonsingular action $\Gamma \actson (X,\mu)$ such that $\cR = \{(g \cdot x,x)\mid g \in \Gamma, x \in X\}$. Define the action groupoid $\cG = \Gamma \times X$ with $\cG^{(0)} = X$, $s(g,x) = x$, $t(g,x) = g \cdot x$ and $(g,h\cdot x) \cdot (h,x) = (gh,x)$. Since 2 holds, there is a groupoid homomorphism $\cR \recht \cG$ lifting the quotient map $\cG \recht \cR$. So there exists a cocycle $\om : \cR \recht \Gamma$ such that $\om(x,y) \cdot y = x$ for a.e.\ $(x,y) \in \cR$. Removing from $X$ a Borel set of measure zero, we may assume that $\om$ is a Borel cocycle and that $\om(x,y)\cdot y = x$ for all $(x,y) \in \cR$.

Define on $X \times \Gamma$ the countable Borel equivalence relation $\cS$ given by $(x,g) \sim_\cS (y,h)$ if and only if $x \sim_\cR y$. Define the subequivalence relation $\cT \subset \cS$ given by $(x,g) \sim_\cT (y,h)$ if and only if $x \sim_\cR y$ and $g = \om(x,y)h$. Since the restriction of $\cT$ to each of the subsets $X \times \{g\}$ is the trivial equivalence relation, it follows that $\cT$ admits a Borel fundamental domain $Y \subset X \times \Gamma$. Define the Borel map $\pi : X \times \Gamma \recht Y$ such that $\pi(x,g) \sim_\cT (x,g)$. Then define the action of $\Gamma$ on $Y$ given by $h \cdot (x,g) = \pi(x,gh^{-1})$. By construction, this action is free and its orbit equivalence relation coincides with $\cS|_Y$. So, $\cS|_Y$ is a treeable Borel equivalence relation. By \cite[Proposition 3.3]{JKL01} (see also \cite[Th\'{e}or\`{e}me 5 and Proposition 2.6]{G99} in the pmp setting), it first follows that $\cS$ is treeable and then that the restriction of $\cS$ to $X \times \{e\}$ is treeable. The latter being equal to $\cR$, the proposition is proven.
\end{proof}

\begin{remark}\label{rem.fixed-price}
Recall that a countable group $\Gamma$ is said to have fixed price if all orbit equivalence relations for free pmp actions of $\Gamma$ have the same cost, in the sense of \cite{G99}. The well known fixed price problem asks whether all countable groups have fixed price.

Let $\cR$ be a countable pmp equivalence relation on $(X,\mu)$. Whenever $\cS$ on $(Y,\eta)$ is the orbit relation for a free pmp action of $\cR$ (equivalently, there is a measure preserving factor map $Y \recht X$ that is bijective on a.e.\ orbit), any graphing for $\cR$ lifts to a graphing of $\cS$, so that $\cost(\cS) \leq \cost(\cR)$. We can therefore ask the following fixed price question for the equivalence relation $\cR$: do we have $\cost(\cS) = \cost(\cR)$ for each such orbit relation $\cS$~?

With precisely the same argument as in the proof of Proposition \ref{prop.all-kinds-treeability}, we get that the following two statements are equivalent.
\begin{enumlist}
\item Every countable group has fixed price.
\item Every countable pmp equivalence relation that is the orbit relation of a free action of a countable group, has fixed price.
\end{enumlist}

So exactly as above, if one believes that there exist countable groups that do not have fixed price, it might be easier to first find countable equivalence relations that do not have fixed price.

Note here as well that in the above context, the $L^2$-Betti numbers of $\cS$ and $\cR$ coincide. Indeed, viewing $\cS$ as the orbit relation of a free pmp action of $\cR$, one may repeat, mutatis mutandis, the proofs of \cite[Corollaire 3.16]{G01} or \cite[Theorem 5.5]{S03} (see also \cite[Remark 9.24]{PSV15} for an alternative proof in terms of Cartan inclusions). So a counterexample for the fixed price question for equivalence relations would also provide a counterexample for the equally wide open problem asking whether $\cost(\cR) = \beta_1^{(2)}(\cR)+1$ for every ergodic countable pmp equivalence relation with infinite orbits.
\end{remark}

\section{(Non-)classification of discrete amenable groupoids}\label{sec.examples-groupoids}

Theorems \ref{thm.classification-amenable} and \ref{thm.classification-amenable-Cartan} provide a complete classification of the regular von Neumann subalgebras $B$ of the hyperfinite II$_1$ factor $R$ satisfying $B' \cap R = \cZ(B)$ in terms of the associated discrete measured groupoid $\cG$. Every amenable discrete measured groupoid is the semi-direct product of a measurable field of amenable groups and an action of a countable amenable equivalence relation. Although there is a unique amenable ergodic equivalence relation of type II$_1$, the classification of amenable discrete groupoids is strictly more complex than the classification of amenable groups, as we show in the following examples. All this is meant in the descriptive set theoretic sense of the word (see e.g.\ \cite{H00}).

\subsection{Measurable fields of amenable groups}

Let $X_0$ be the uncountable standard Borel space and assume that we are given a Borel family $(\Gamma_x)_{x \in X_0}$ of amenable groups with the property that $\Gamma_x \not\cong \Gamma_y$ if $x \neq y$. More concretely, one could take $X_0 = \{0,1\}^\cP$, where $\cP$ denotes the set of prime numbers, and define for $x \in X_0$,
$$\Gamma_x = \bigoplus_{p \in \cP, x_p = 1} \Z / p \Z \; .$$
For every probability measure $\mu$ on $X_0$, we define $\cG(\mu)$ as the direct integral w.r.t.\ $\mu$ of the amenable groups $(\Gamma_x)_{x \in X_0}$. By construction, $\cG(\mu_1) \cong \cG(\mu_2)$ if and only if $\mu_1$ and $\mu_2$ are absolutely continuous. By \cite[Theorem 2.1]{KS99}, the nonatomic probability measures on $X_0$ (up to absolute continuity) are not classifiable by countable structures. Countable amenable groups form, by definition, a countable structure. In this sense, already these trivial examples show that the classification of amenable discrete measured groupoids is strictly harder than the classification of amenable groups.

\subsection{Ergodic amenable discrete groupoids}

In order to arise from regular inclusions $B \subset R$ with $R$ being the hyperfinite II$_1$ \emph{factor}, we need to consider ergodic groupoids $\cG$. Unclassifiably many ergodic amenable discrete groupoids can be constructed as follows. Again take $X_0 = \{0,1\}^\cP$, where $\cP$ denotes the set of prime numbers. We now define, for $x \in X_0$, the subgroup $\Lambda_x \subset \Q$ given by
$$\Lambda_x = \langle p^{-k} \mid p \in \cP, x_p = 1 , k \in \N \rangle \; .$$
One checks that for $x \neq y$, every group homomorphism $\Lambda_x \recht \Lambda_y$ is trivial. Then define $X = X_0^\Z$ and for every $x \in X$, put
$$\Gamma_x = \bigoplus_{n \in \Z} \Lambda_{x_n} \; .$$
Consider the Bernoulli action $\Z \actson X$ given by $(k \cdot x)_n = x_{n-k}$ and denote by $\delta_{k,x} : \Gamma_x \recht \Gamma_{k \cdot x}$, the natural group isomorphism.
We then define the Borel groupoid $\cG_0$ as the union of all $(\Gamma_x)_{x \in X}$ with unit space $X$, and we define $\cG$ as the semidirect product of $\cG_0$ and $\Z$, acting by the automorphisms that we just defined. For every probability measure $\mu$ on $X_0$, the product measure $\mu^\Z$ on $X$ is $\Z$-invariant, so that we obtain the ergodic, pmp, amenable, discrete measured groupoid $\cG(\mu)$.

We claim that again $\cG(\mu_1) \cong \cG(\mu_2)$ if and only if the measures $\mu_1$ and $\mu_2$ are absolutely continuous. So also the discrete measured groupoids that are ergodic, pmp and amenable are unclassifiable by countable structures. To prove the claim, it suffices to show the following statement: if $\cU \subset X_0$ is a Borel set such that $\mu_1(\cU) > 0$ and $\mu_2(\cU) = 0$, then for $\mu_1^\Z$-a.e.\ $x \in X$ and $\mu_2^\Z$-a.e.\ $y \in X$, we have that $\Gamma_x \not\cong \Gamma_y$. Fix such a Borel set $\cU \subset X_0$. For $\mu_1^\Z$-a.e.\ $x \in X$, we have that there exists an $n \in \Z$ with $x_n \in \cU$, so that one of the groups $\Lambda_a$, $a \in \cU$, is a quotient of $\Gamma_x$. For $\mu_2^\Z$-a.e.\ $y \in X$, we have that $y_n \not\in \cU$ for all $n \in \Z$, so that none of the groups $\Lambda_a$, $a \in \cU$, is a quotient of $\Gamma_y$. Then the claim is proven.

\subsection{Amenable discrete groupoids with constant isotropy groups}

In all the examples so far, we obtain ``many'' groupoids by exploiting that there are ``many'' amenable groups to choose from as isotropy groups $\Gamma_x$. But even the ergodic pmp amenable discrete groupoids $\cG$ with constant isotropy groups $\Gamma_x \cong H$ for a fixed amenable group $H$ can be unclassifiable. This however depends on the choice of $H$, as we prove in Corollary \ref{cor.classif-vs-non-classif}: when $H = \Z$, there are only countably many such groupoids $\cG$ up to isomorphism, while for $H = \Z^2$, we encode ergodic transformations of the interval $[0,1]$ up to conjugacy into such amenable groupoids up to isomorphism.

Fix a countable nonsingular equivalence relation $\cR$ on the standard probability space $(X,\mu)$ and fix a countable group $H$. For every $1$-cocycle
$$\delta : \cR \recht \Aut(H) : (x,y) \mapsto \delta_{(x,y)} \; ,$$
define the groupoid $\cG_\delta = H \times \cR$ with $\cG_\delta^{(0)} = X$ and
$$s(g,x,y) = y \;\; , \;\; t(g,x,y) = x \;\; \text{and} \;\; (g,x,y) \cdot (h,y,z) = (g \delta_{(x,y)}(h) , x,z) \; .$$
Note that $\cG_\delta$ is amenable if and only if both $H$ and $\cR$ are amenable.

Conversely, every discrete measured groupoid $\cG$ with the properties that the isotropy groups $\Gamma_x$ are isomorphic with a fixed group $H$ for a.e.\ $x \in X = \cG^{(0)}$ and that the associated equivalence relation on $(X,\mu)$ is amenable, is isomorphic with $\cG_\delta$ for some $1$-cocycle $\delta : \cR \recht \Aut(H)$.

As in \cite[Definition 1.2]{BG90}, given a Polish group $G$, the $1$-cocycles $\delta, \delta' : \cR \recht G$ are called \emph{weakly equivalent} if there exists a $\Delta \in \Aut(\cR)$ such that $\delta$ is cohomologous with $\delta' \circ \Delta$. More concretely, $\Delta$ and $\Delta'$ are weakly equivalent if and only if there exists a nonsingular automorphism $\Delta \in \Aut(X,\mu)$ and a Borel map $\vphi : X \recht G$ such that
\begin{align*}
& \Delta(\cR \cdot x) = \cR \cdot \Delta(x) \;\;\text{for a.e.\ $x \in X$, and}\\
& \delta'(\Delta(x),\Delta(y)) = \vphi(x) \, \delta(x,y) \, \vphi(y)^{-1}\;\;\text{for a.e.\ $(x,y) \in \cR$.}
\end{align*}

\begin{proposition}\label{prop.isom-vs-weak-equiv}
Let $\cR$ be an amenable countable nonsingular equivalence relation and let $H$ be a countable group. Define the Polish group $G = \Aut(H)/\overline{\Inn(H)}$ and denote by $p : \Aut(H) \recht G$ the quotient map.
\begin{enumlist}
\item Let $\delta, \delta' : \cR \recht \Aut(H)$ be $1$-cocycles. Then $\cG_\delta$ is isomorphic with $\cG_{\delta'}$ if and only if the $G$-valued $1$-cocycles $p \circ \delta$ and $p \circ \delta'$ are weakly equivalent.
\item For every $1$-cocycle $\delta_0 : \cR \recht G$, there exists a $1$-cocycle $\delta : \cR \recht \Aut(H)$ satisfying $\delta_0 = p \circ \delta$.
\end{enumlist}
\end{proposition}
\begin{proof}
1.\ A direct computation gives that $\cG_\delta \cong \cG_{\delta'}$ if and only if there exists a $\Delta \in \Aut(\cR)$ and Borel maps
$$\om : \cR \recht H : (x,y) \mapsto \om(x,y) \quad\text{and}\quad \theta : X \recht \Aut(H) : x \mapsto \theta_x$$
such that
\begin{align}
& \om(x,z) = \om(x,y) \, \delta_{(x,y)}(\om(y,z)) \quad\text{for a.e.\ $(x,y),(y,z) \in \cR$,} \label{first.prop}\\
& \Ad \om(x,y) \circ \delta_{(x,y)} = \theta_x \circ \delta'_{(\Delta(x),\Delta(y))} \circ \theta_y^{-1} \quad\text{for a.e.\ $(x,y) \in \cR$,}\label{second.prop}
\end{align}
and with the isomorphism $\theta : \cG_{\delta'} \recht \cG_\delta$ given by $\theta(g,\Delta(x),\Delta(y)) = (\theta_x(g) \, \om(x,y) , x, y)$. So if $\cG_\delta \cong \cG_{\delta'}$, we immediately get that $p \circ \delta$ and $p \circ \delta'$ are weakly equivalent.

Conversely, assume that $p \circ \delta$ and $p \circ \delta'$ are weakly equivalent. Since $\cR$ is amenable, we can write, up to measure zero, $X = X_1 \sqcup X_2$ so that the restriction of $\cR$ to $X_1$ is of type I, while the restriction of $\cR$ to $X_2$ is the orbit relation of a free action of $\Z$ implemented by $T \in \Aut(X_2,\mu)$. Every $1$-cocycle for $\cR|_{X_1}$ is cohomologous to the trivial $1$-cocycle, so that the restrictions of $\cG_\delta$ and $\cG_{\delta'}$ to $X_1$ are both isomorphic with the direct product of $H$ and $\cR|_{X_1}$. We may thus assume that $X = X_2$.

Since $p \circ \delta$ and $p \circ \delta'$ are weakly equivalent, using a Borel lift $G \recht \Aut(H)$ for the quotient homomorphism $p$, we find an automorphism $\Delta \in \Aut(\cR)$ and a Borel map $\theta : X \recht \Aut(H) : x \mapsto \theta_x$ such that the $1$-cocycle $\delta\dpr$ defined by
$$\delta\dpr(x,y) = \theta_x \circ \delta'_{(\Delta(x),\Delta(y))} \circ \theta_y^{-1}$$
satisfies $p \circ \delta = p \circ \delta\dpr$. By \cite[Appendix]{Jones-Takesaki-82}, because $\cR$ is amenable, we can modify $\theta$ and choose a Borel map $\eta : X \recht H$ such that
\begin{equation}\label{eq.we-need}
\Ad \eta(x) \circ \delta_{(Tx,x)} = \theta_{Tx} \circ \delta'_{(\Delta(Tx),\Delta(x))} \circ \theta_x^{-1}
\end{equation}
for a.e.\ $x \in X$. Uniquely define $\om : \cR \recht H$ such that \eqref{first.prop} holds and $\om(Tx,x) = \eta(x)$ for all $x \in X$. Then \eqref{second.prop} follows from \eqref{eq.we-need} and we obtain that $\cG_\delta \cong \cG_{\delta'}$.

2.\ Let $\delta_0 : \cR \recht G$ be a $1$-cocycle. Also here, we may assume that $\cR$ is the orbit relation of a free action of $\Z$ implemented by $T \in \Aut(X,\mu)$. Using a Borel lift $G \recht \Aut(H)$, choose a Borel function $\eta : X \recht \Aut(H)$ such that $p(\eta(x)) = \delta_0(Tx,x)$ for a.e.\ $x \in X$. Define $\delta : \cR \recht \Aut(H)$ as the unique $1$-cocycle satisfying $\delta(Tx,x) = \eta(x)$ for all $x \in X$. It follows that $p \circ \delta = \delta_0$.
\end{proof}

To every $1$-cocycle $\delta : \cR \recht G$ with values in a lcsc group $G$ is associated the \emph{Mackey action,} defined as follows. We equip $G$ with its left or right Haar measure and consider on $X \times G$ the equivalence relation $\cR_\delta$ given by $(x,g) \sim_{\cR_\delta} (y,h)$ if and only if $x \sim_\cR y$ and $g = \delta(x,y) h$. The group $G$ acts by automorphisms of $\cR_\delta$ given by $g \cdot (x,h) = (x,hg^{-1})$. We define the nonsingular action $G \actson (Y,\eta)$ in such a way that $L^\infty(Y)$ can be identified with the $\cR_\delta$-invariant functions in $L^\infty(X \times G)$. If $\cR$ is ergodic, the action $G \actson (Y,\eta)$ is ergodic. It is called the \emph{Mackey action} of the $1$-cocycle $\delta$. Clearly, weakly equivalent $1$-cocycles give rise to conjugate Mackey actions. When $\cR$ is amenable and pmp, by \cite[Theorem 3.1]{GS91}, the converse holds for recurrent $1$-cocycles. Here, $\delta$ is called recurrent if $\cR_\delta$ is not of type I. In this way, we get the following corollary to Proposition \ref{prop.isom-vs-weak-equiv}.

\begin{corollary}\label{cor.classif-vs-non-classif}
Fix a countable group $H$. Consider all discrete measured groupoids $\cG$ such that almost every isotropy group is isomorphic with $H$ and such that the associated equivalence relation $\cR$ is the unique ergodic hyperfinite type II$_1$ equivalence relation. Write $G = \Aut(H) / \overline{\Inn(H)}$.
\begin{enumlist}
\item If $G$ is a compact group, the groupoids $\cG$ are concretely classifiable in terms of the space of closed subgroups of $G$ up to conjugacy.
\item If $G$ is discrete and contains a copy of $\Z$, the groupoids $\cG$ are not classifiable by countable structures: to every weakly mixing measure preserving transformation $T \in \Aut([0,1])$, we can associate such a groupoid $\cG_T$ such that $\cG_T \cong \cG_S$ if and only if the transformations $T,S$ are (flip) conjugate.
\end{enumlist}
\end{corollary}

Note that we are in the first situation when $H = \Z$ or when $H$ is a finite group, while we are in the second situation when $H = \Z^n$ with $n \geq 2$.

\begin{proof}
1.\ By Proposition \ref{prop.isom-vs-weak-equiv}, the classification of the groupoids $\cG$ is equivalent with the classification of the $1$-cocycles $\delta : \cR \recht G$ up to weak equivalence. The Mackey action of such a $1$-cocycle is an ergodic action of the compact group $G$ and thus of the form $G \actson G/K$ for some closed subgroup $K < G$. Two such actions are conjugate if and only if the corresponding closed subgroups are conjugate. Using the canonical $K$-valued $1$-cocycle for the action $\Z \actson K^\Z / K$, where $K$ is sitting diagonally in $K^\Z$ and $\Z$ is acting as a shift, every action $G \actson G/K$ arises as the Mackey action of a $1$-cocycle. Finally, for every $1$-cocycle $\delta : \cR \recht G$, the equivalence relation $\cR_\delta$ preserves a probability measure, so that every $1$-cocycle $\delta : \cR \recht G$ is recurrent. By the discussion above, using \cite[Theorem 3.1]{GS91}, the result follows.

2.\ Denote by $p : \Aut(H) \recht G$ the quotient homomorphism. Choose $\al \in \Aut(H)$ such that $a := p(\al)$ is an element of infinite order in $G$. Fix a standard nonatomic probability space $(X,\mu)$. For every weakly mixing pmp transformation $T \in \Aut(X,\mu)$, consider the orbit equivalence relation $\cR_T \cong \cR$ of the associated $\Z$-action and consider the $1$-cocycle $\delta_T : \cR_T \recht \Aut(H)$ given by $\delta_T(T^n x,x) = \al^n$ for all $x \in X$, $n \in \Z$. Denote by $\cG_T$ the associated groupoid.

We distinguish two cases.
\begin{enumlist}[label=\alph*.]
\item There exists an element $g \in G$ such that $g a g^{-1} = a^{-1}$.
\item There is no element $g \in G$ such that $g a g^{-1} = a^{-1}$.
\end{enumlist}
We prove that in case a, the groupoids $\cG_T$ and $\cG_S$, for weakly mixing pmp transformations $T,S \in \Aut(X,\mu)$, are isomorphic if and only if there exists a $\Delta \in \Aut(X,\mu)$ such that $\Delta T \Delta^{-1} = S^{\pm 1}$. In case b, we prove that $\cG_T \cong \cG_S$ if and only if there exists a $\Delta \in \Aut(X,\mu)$ such that $\Delta T \Delta^{-1} = S$.

When $T$ and $S$ are conjugate, it is immediate that $\cG_T \cong \cG_S$. When $\Delta T \Delta^{-1} = S^{-1}$ and $g a g^{-1} = a^{-1}$, it follows that $\Delta$ is an isomorphism of $\cR_T$ onto $\cR_S$ and that
$$g \, (p \circ \delta_S \circ \Delta)(x,y) \, g^{-1} = (p \circ \delta_T)(x,y)$$
for all $(x,y) \in \cR_T$. So, the $1$-cocycles $p \circ \delta_S$ and $p \circ \delta_T$ are weakly equivalent. It follows from Proposition \ref{prop.isom-vs-weak-equiv} that $\cG_T \cong \cG_S$.

Conversely, assume that $\cG_T \cong \cG_S$. By Proposition \ref{prop.isom-vs-weak-equiv}, the $1$-cocycles $p \circ \delta_S$ and $p \circ \delta_T$ are weakly equivalent. So, their Mackey actions are conjugate. These Mackey actions are precisely the induced actions $G \actson G \times_S X$, resp.\ $G \actson G \times_T X$, of $\Z \actson (X,\mu)$ given by $S$, resp.\ $T$, where $\Z$ is identified with the subgroup of $G$ generated by $a$. Denote by $\Delta : G \times_S X \recht G \times_T X$ the conjugacy of the Mackey actions. Since $S$ and $T$ are weakly mixing transformations, we must have $\Delta(\{e\} \times X) = \{g\} \times X$ with $g^{-1} a^\Z g = a^\Z$. If $g^{-1} a g = a$, it follows that $S$ and $T$ are conjugate. If $g^{-1} a g = a^{-1}$, it follows that $S$ and $T^{-1}$ are conjugate.

By \cite{FW03} (see also \cite[Theorem 5.7.b]{K10}), the classification of weakly mixing pmp transformations up to (flip) conjugacy is not classifiable by countable structures. So the groupoids $\cG$ are not classifiable by countable structures.
\end{proof}

\subsection{Exotic examples}

Let $\cR$ be the unique hyperfinite equivalence relation of type II$_1$. In this section, we construct for every ergodic subequivalence relation $\cR_0 \subset \cR$ an ergodic pmp amenable discrete groupoid $\cG(\cR_0)$ and prove that
$\cG(\cR_0) \cong \cG(\cR_1)$ if and only if there exists $\vphi \in [\cR]$ such that $(\vphi \times \vphi)(\cR_0) = \cR_1$. So, the classification of this family $\cG(\cR_0)$ of amenable discrete measured groupoids is equivalent with the classification, up to conjugacy by an element of the full group, of all ergodic subequivalence relations of the hyperfinite type II$_1$ ergodic equivalence relation $\cR$.

Put $X = \{0,1\}^\cP$ and denote by $\mu$ the product of the uniform probability measure on $\{0,1\}$. View $X$ as the space of subsets of the set $\cP$ of prime numbers. For every $\cQ \in X$, define the subgroup $\Gamma_\cQ \subset \Q$ given by
$$\Gamma_\cQ = \langle p^{-1} \mid p \in \cQ \rangle \; .$$
By construction, $\Gamma_\cQ \cong \Gamma_{\cQ'}$ if and only if $|\cQ \vartriangle \cQ'| < \infty$. Realize $\cR$ as the countable pmp equivalence relation on $(X,\mu)$ consisting of all $(\cQ,\cQ')$ with $|\cQ \vartriangle \cQ'| < \infty$. We construct for every $(\cQ,\cQ') \in \cR$ a canonical isomorphism $\delta_{(\cQ,\cQ')} : \Gamma_{\cQ'} \recht \Gamma_{\cQ}$. Define the element $q(\cQ,\cQ') \in \Q^*$ given by
$$q(\cQ,\cQ') = \Bigl( \prod_{p \in \cQ \setminus \cQ'} p^{-1} \Bigr) \cdot \Bigl( \prod_{p \in \cQ' \setminus \cQ} p \Bigr) \; .$$
Then define $\delta_{(\cQ,\cQ')}$ given by multiplication with $q(\cQ,\cQ')$. One checks that
$$\delta_{(\cQ,\cQ\dpr)} = \delta_{(\cQ,\cQ')} \circ \delta_{(\cQ',\cQ\dpr)} \quad\text{for all $(\cQ,\cQ'), (\cQ',\cQ\dpr) \in \cR$.}$$
So we can define the discrete measured groupoid $\cG$ as the semidirect product of the field $(\Gamma_\cQ)_{\cQ \in X}$ and the equivalence relation $\cR$, with the invariant probability measure $\mu$.

For every ergodic subequivalence relation $\cR_0 \subset \cR$, we denote by $\cG(\cR_0)$ the subgroupoid of $\cG$ given by only taking the semidirect product with $\cR_0$.
Since $\Gamma_\cQ \cong \Gamma_{\cQ'}$ if and only if $(\cQ,\cQ') \in \cR$, the family of groupoids $\cG(\cR_0)$ has the properties mentioned above.

\end{document}